\documentclass[reqno]{amsart}
\usepackage[svgnames,dvipsnames]{xcolor}
\usepackage{geometry}
\usepackage[english]{babel}
\usepackage{latexsym, mathrsfs, verbatim}
\usepackage{amsfonts, amsthm, amsmath, amssymb, bm}
\usepackage{hyperref}
\usepackage{amsaddr}
\usepackage[nosepfour,autolanguage]{numprint}
\usepackage[T1]{fontenc}
\usepackage[utf8]{inputenc} 

\usepackage[linesnumbered,ruled]{algorithm2e}

\usepackage{tikz}
\usetikzlibrary{mindmap,backgrounds,shapes,calc,intersections,rulercompass}

\usepackage{subcaption}

\newtheorem{prop}{Proposition}[section]
\newtheorem{lem}[prop]{Lemma}
\newtheorem{cor}[prop]{Corollary}
\newtheorem{thm}[prop]{Theorem}

\theoremstyle{definition}
\newtheorem{rem}[prop]{Remark}
\newtheorem{defi}[prop]{Definition}
\newtheorem{ex}[prop]{Example}
\newtheorem{prob}[prop]{Problem}

\def\ga{\bm{\gamma}}

\def\SRG{\textsc{srg}}
\def\DDG{\textsc{ddg}}
\def\diam{\mathrm{diam}}
\def\c{\mathrm{c}}

\numberwithin{equation}{section}

 \usepackage[normalem]{ulem}
 \usepackage{datetime}

\hypersetup{ 
 colorlinks =true,
 linkcolor=Sepia,
 citecolor=Plum,
 urlcolor=olive,
 }

\begin{document}

\title[A graph reconstruction problem involving common neighbors]{A graph reconstruction problem \\ involving common neighbors}
\author[M. Ascolese]{Michela Ascolese}

\author[P. Negrini, S.M.C. Pagani \and M.A. Pellegrini]{Pietro Negrini, Silvia Maria Carla Pagani \\ \and Marco Antonio Pellegrini}
\address{Dipartimento di Matematica e Fisica, Universit\`a Cattolica del Sacro Cuore,\\
Via della Garzetta 48, 25133 Brescia, Italy}
\email{pietro.negrini@unicatt.it}
\email{silvia.pagani@unicatt.it}
\email{marcoantonio.pellegrini@unicatt.it}

\begin{abstract}
Given a simple graph $G = (V, E)$ on $v$ vertices and two distinct vertices $x, y \in  V$, the co-degree $c_{x,y}$ associated to the pair $\{x, y\}$ is the number of their common neighbors in the graph $G$.
The \emph{co-degree sequence} of $G$, denoted by $\ga(G)$, is the list of all the co-degrees associated to all the possible pairs of distinct vertices, arranged in non-increasing order.
In this paper we consider the following problem, which can be viewed as a generalization of a result by Erd\H{o}s and Gallai as well as of the Erd\H{o}s, Rényi and Sós' friendship theorem: 
given an integer $v\geqslant 2$ and a sequence $\ga$  of nonnegative integers arranged in non-increasing order,
establish if there exists a simple graph on $v$ vertices having $\ga$ as its co-degree sequence and, in case of positive answer, provide such a graph.
We provide a full answer to this problem for the class of planar $C_4$-free graphs. 
\end{abstract}

\maketitle

\section{Introduction}

A simple graph $G$ is given by a pair $(V, E)$, where the elements of $V$ are called vertices, while $E$ is a set of edges, i.e., unordered pairs $\{x,y\}$ of distinct vertices.
In this paper, we only consider simple graphs with a finite number $v$ of vertices.

The \emph{degree} $\deg_G(x)$ of a vertex $x$ is defined as the number of edges in which the vertex appears, while the vector $\bm{\pi} = (d_1 , \ldots, d_v )$ of all
degrees, arranged in non-increasing order,  is called the degree sequence of the graph $G$. 
One can ask when a sequence of positive integers may be the degree sequence of a simple graph $G$.
In this case, we say that $G$ realizes $\bm{\pi}$. This question was answered by Erd\H{o}s and Gallai in 1960 \cite{EG} (several other equivalent characterizations were provided, see \cite{SH} for a survey).

In this paper we consider a similar problem, based on the concept of \emph{co-degree}.
Given a simple graph $G = (V, E)$ on $v$ vertices and two distinct vertices $x, y \in  V$, the co-degree $c_{x,y}$ associated to the pair $\{x, y\}$ is the number of their common neighbors in the graph $G$: $c_{x,y} = |N_G (x) \cap N_G (y)|$, where $N_G(z)$ denotes the \emph{neighborhood} of a vertex $z\in V$, namely, the set of vertices adjacent to $z$.
The \emph{co-degree sequence} of $G$, denoted by $\ga(G)$, is the list of all the co-degrees associated to all the possible pairs of distinct vertices, arranged in non-increasing order. 
Its length is $\binom{v}{2}$, and it contains null entries if and only if at least one a pair of vertices does not share any neighbor.
The co-degrees can be computed also from the adjacency matrix $A$ of the graph $G$. In fact, $c_{x,y}$ is the entry of $A^2$ corresponding to the pair of vertices $\{x,y\}$.

An interesting and challenging problem is  to understand how much information about a simple graph can be obtained starting from its co-degree sequence.

\begin{prob}\label{problema}
Given an integer $v\geqslant 2$ and a sequence $\ga$ of length $\binom{v}{2}$ of nonnegative integers arranged in non-increasing order, establish if there exists a simple graph on $v$ vertices having $\ga$ as its co-degree sequence and, in case of positive answer, provide such a graph.
\end{prob}

After having highlighted in Section~\ref{deza} some connections between our problem and important graph classes, such as strongly regular graphs, Deza graphs and  $(0,\lambda)$-graphs, in Section~\ref{free4} we focus our attention on $C_4$-free graphs.
In particular, in Section~\ref{alberi} we address the problem for the family of trees on $v$ vertices. 
Section~\ref{planari} is devoted to the solution of Problem~\ref{problema} for $C_4$-free planar graphs. Finally, Section \ref{sec:conclusioni} proposes some further work and concludes the paper.

\subsection*{Notation}

Given $c\in \{1,2\}$ and two integers $a \leqslant b$ such that $a\equiv b \pmod c$, we denote by $[a, b]_c$ the set consisting of the integers $a, a + c, a+2c, \ldots, b$. 
If $a > b$, then $[a, b]_c$ is empty. 
We simply write $[a,b]$ for $[a,b]_1$.
If a sequence $L$ of integers consists of $a_0$ $0$'s, $a_1$ $1$'s, \ldots, $a_k$ $k$'s, we will write 
$L = \left( k^{a_k}, \ldots, 1^{a_1} , 0^{a_0}\right)$, often omitting the values with exponent zero.
A sequence $(\alpha_1,\alpha_2,\ldots,\alpha_k)$ of positive integers is said to be a partition of an integer $m$ if $\alpha_1\geqslant \alpha_2\geqslant \cdots \geqslant \alpha_k$ and $\sum\limits_{i=1}^k \alpha_i =m$.

We denote by $K_v$ the complete graph on $v$ vertices, and by $K_{a,b}$ the complete bipartite graph consisting of two parts of respective cardinality $a$ and $b$.
Given a simple graph $G$, $\diam(G)$ will denote its diameter and $\Delta(G)$ will denote its maximum degree:
$\Delta(G)=\max\limits_{x \in V} \deg_G(x)$.

\section{Co-degrees and notable classes of graphs}\label{deza}

In this section, we highlight some connections between co-degree  sequences and certain classes of graphs. In particular, we consider graphs whose co-degree sequence consists of at most two different integers.

\begin{defi}\cite{AS}\label{la}
A simple graph $G=(V,E)$ is said to be a \emph{symmetric $(v, k, \lambda)$ graph}, where  $\lambda, v ,k$ are three positive integers such that $\lambda < k < v$, if the following conditions are satisfied:
\begin{itemize}
\item[$(1)$] $G$ is  $k$-regular on $v$ vertices;
\item[$(2)$] every pair of vertices has exactly $\lambda$ common neighbors.
\end{itemize}
\end{defi}

If $V = \{x_1,\ldots,x_v\}$ and $S_i=N_G(x_i)$,  $i = 1,\ldots, v$, then the sets $S_i$ form a symmetric block design with parameters $(v, k, \lambda)$, i.e., $|S_i| = k$ and $|S_i \cap S_j|=\lambda$  for $i\neq j$.
As remarked in \cite{R}, the study of symmetric $(v,k,\lambda)$ graphs is equivalent to the study of $(v,k,\lambda)$ designs  admitting polarities with no absolute points. 
As follows for instance from \cite{ERS}, the only symmetric $(v, k, 1)$ graph is the complete graph $K_3$. 
Considering only Condition $(2)$ of Definition~\ref{la} (i.e., dropping the hypothesis of $k$-regularity), Erd\H{o}s, Rényi and Sós proved the following result.

\begin{thm}\cite{ERS}\label{friendship}
If $G$ is a simple graph on $v$ vertices in which any two vertices have exactly one common neighbor, then $v$ is odd and $G$ consists of $\frac{v-1}{2}$ triangles which have one common vertex.
\end{thm}

The graph described in the previous theorem is the so called \emph{friendship graph}, see Figure~\ref{amici}.

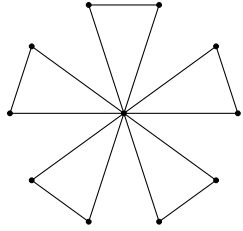
\begin{figure}[ht]
\begin{tikzpicture}[scale = 0.3]

\draw 
  node (p) [
    minimum size=3cm,
    regular polygon, 
    regular polygon sides=10
  ] {};
  
\foreach \n in {1,2,...,10}{
 \node[anchor=\n*(360/10)] at (p.corner \n) {};
}
\coordinate (A) at (p.corner 1);
\coordinate (B) at (p.corner 2);
\coordinate (C) at (p.corner 3);
\coordinate (D) at (p.corner 4);
\coordinate (E) at (p.corner 5);
\coordinate (F) at (p.corner 6);
\coordinate (G) at (p.corner 7);
\coordinate (H) at (p.corner 8);
\coordinate (I) at (p.corner 9);
\coordinate (J) at (p.corner 10);
\coordinate (O) at (0,0);

\draw (O)--(A)--(B)--(O)--(C)--(D)--(O)--(E)--(F)--(O)--(G)--(H)--(O)--(I)--(J)--(O);

\draw[fill] (A) circle (3pt);
\draw[fill] (B) circle (3pt);
\draw[fill] (C) circle (3pt);
\draw[fill] (D) circle (3pt);
\draw[fill] (E) circle (3pt);
\draw[fill] (F) circle (3pt);
\draw[fill] (G) circle (3pt);
\draw[fill] (H) circle (3pt);
\draw[fill] (I) circle (3pt);
\draw[fill] (J) circle (3pt);
\draw[fill] (O) circle (3pt);

\end{tikzpicture}
\caption{The friendship graph on $v=11$ vertices.}\label{amici}
\end{figure}

On the other hand, Condition $(2)$ of Definition~\ref{la}  for $\lambda\geqslant 2$ implies the regularity of $G$, as shown by Bose and  Shrikhande.

\begin{lem}\cite{BS}
Let $G$ be a simple graph on $v$ vertices in which each pair of distinct vertices is adjacent to exactly $\lambda$ other vertices, $\lambda\geqslant 2$.
Then $G$ is regular of valence $k$ such that $v-1=\frac{k(k-1)}{\lambda}$ and there exists a positive integer $m$ such that $k = \lambda+m^2$ and $\frac{\lambda}{m}$ is an integer with the same parity as $v-1-m$.
\end{lem}

We can reformulate the previous results in terms of co-degree sequences. Clearly, the co-degree sequence of a symmetric $(v,k,\lambda)$ graph is $\ga=\left(\lambda^{\binom{v}{2}}\right)$.

\begin{prop}
Let $G$ be a simple graph on $v\geqslant 3$ vertices. Then,
\begin{itemize}
\item[$(1)$] $\ga(G)=\left(1^{\binom{v}{2}}\right)$ if and only if $v$ is odd and $G$ is the  friendship graph;
\item[$(2)$]  $\ga(G)=\left(\lambda^{\binom{v}{2}}\right)$ with $\lambda \geqslant 2$ if and only if $G$ is a symmetric $(v,k,\lambda)$ graph, where $v-1=\frac{k(k-1)}{\lambda}$;
\item[$(3)$] $\ga(G)=\left((v-2)^{\binom{v}{2}}\right)$ if and only if $G$ is the complete graph $K_v$.
\end{itemize}
\end{prop}

\begin{proof}
By the above, it suffices to prove item (3). Clearly, the co-degree sequence of the complete graph $K_v$  is $\left((v-2)^{\binom{v}{2}}\right)$. On the other hand, let $x$ and $y$ be two any distinct vertices of a simple graph $G$ such that $\ga(G)=\left((v-2)^{\binom{v}{2}}\right)$. Take another vertex $z \in V \setminus \{x,y\}$, which exists since $v \geqslant 3$. Since $c_{x,z} = v-2$ and $G$ is simple, $x$ and $y$ are adjacent. As the choice of $x$ and $y$ was arbitrary, $G$ is the complete graph on $v$ vertices.
\end{proof}

We recall that a \emph{strongly regular graph} $\SRG(v,k,\lambda,\mu)$ is a $k$-regular graph $G$ on $v$ vertices such that, given two vertices $x$ and $y$, one has 
$$ c_{x,y}=\left\{\begin{array}{ll}
   \lambda   &  \text{if $x$ and $y$ are adjacent},\\
   \mu & \text{otherwise}.
 \end{array}\right.$$
So, a symmetric $(v,k,\lambda)$ graph is nothing but a $\SRG(v,k,\lambda,\lambda)$.
For  a strongly regular graph $\SRG(v,k,\lambda,\mu)$ one has
$$\ga=\left\{
\begin{array}{ll}
\left(\lambda^e, \mu ^{\binom{v}{2}-e}\right) & 
 \text{if } \lambda\geqslant 
        \mu,\\
    \left(\mu ^{\binom{v}{2}-e},\lambda^e\right) & 
        \text{otherwise},
    \end{array}\right.$$
where $e=\frac{vk}{2}$ is the number of its edges.
    
The concept of strongly regular graph can be generalized as follows, see \cite{EFHHH}.
Let $v$, $k$, $b$, $a$ be integers such that $0 \leqslant  a \leqslant b \leqslant  k < v$. 
A simple graph $G=(V,E)$ is called a \emph{Deza graph} with parameters $(v, k, b, a)$ if it has $v$ vertices, is $k$-regular and $\{c_{x,y}  : x,y\in V\}=\{a,b\}$.
Figure~\ref{C4free} shows a Deza graph with parameters $(12,3,1,0)$ and co-degree sequence $(1^{36},0^{30})$.

\begin{figure}[ht]
\centering
\begin{subfigure}{0.3\textwidth}
\centering
\begin{tikzpicture}[scale = 0.3]
\draw 
  node (p) [
    minimum size=1cm,
    regular polygon, 
    regular polygon sides=5
  ] {};
  
\foreach \n in {1,2,...,5}{
 \node[anchor=\n*(360/5)] at (p.corner \n) {};
}
\coordinate (A) at (p.corner 1);
\coordinate (B) at (p.corner 2);
\coordinate (C) at (p.corner 3);
\coordinate (D) at (p.corner 4);
\coordinate (E) at (p.corner 5);

\draw 
  node (q) [
    minimum size=3cm,
    regular polygon, 
    regular polygon sides=5
  ] {};
  
\foreach \n in {1,2,...,5}{
 \node[anchor=\n*(360/5)] at (p.corner \n) {};
}
\coordinate (F) at (q.corner 1);
\coordinate (G) at (q.corner 2);
\coordinate (H) at (q.corner 3);
\coordinate (I) at (q.corner 4);
\coordinate (J) at (q.corner 5);
\coordinate (K) at (2.5,-1.35);
\coordinate (L) at (-2.5,-1.35);

\draw (A)--(B)--(C)--(D)--(E)--(A)--(F)--(G)--(H)--(I)--(J)--(F);
\draw (C)--(I);
\draw (J)--(K)--(D);
\draw (K)--(E);
\draw (H)--(L)--(B);
\draw (G)--(L);

\draw[fill] (A) circle (3pt);
\draw[fill] (B) circle (3pt);
\draw[fill] (C) circle (3pt);
\draw[fill] (D) circle (3pt);
\draw[fill] (E) circle (3pt);
\draw[fill] (F) circle (3pt);
\draw[fill] (G) circle (3pt);
\draw[fill] (H) circle (3pt);
\draw[fill] (I) circle (3pt);
\draw[fill] (J) circle (3pt);
\draw[fill] (K) circle (3pt);
\draw[fill] (L) circle (3pt);

\end{tikzpicture}
\caption{$G_1$.}\label{C4free}
\end{subfigure}
\begin{subfigure}{0.3\textwidth}
\centering
\begin{tikzpicture}[scale = 0.3]

\draw 
  node (p) [
    minimum size=1.5cm,
    regular polygon, 
    regular polygon sides=5
  ] {};
  
\foreach \n in {1,2,...,5}{
 \node[anchor=\n*(360/5)] at (p.corner \n) {};
}
\coordinate (A) at (p.corner 1);
\coordinate (B) at (p.corner 2);
\coordinate (C) at (p.corner 3);
\coordinate (D) at (p.corner 4);
\coordinate (E) at (p.corner 5);

\draw[fill] (A) circle (3pt);
\draw[fill] (B) circle (3pt);
\draw[fill] (C) circle (3pt);
\draw[fill] (D) circle (3pt);
\draw[fill] (E) circle (3pt);

\draw (A)--(B)--(C)--(D)--(E)--(A)--(C);
\draw (A)--(D);
\end{tikzpicture}
\caption{$G_2$.}\label{221111}
\end{subfigure}
\begin{subfigure}{0.3\textwidth}
\centering
\begin{tikzpicture}[scale = 0.3]

\coordinate  (A) at (0,0);
\coordinate  (B) at (0,6);
\coordinate  (C) at (6,6);
\coordinate  (D) at (6,0);
\coordinate  (E) at (2,2);
\coordinate  (F) at (2,4);
\coordinate  (G) at (4,4);
\coordinate  (H) at (4,2);

\draw (A)--(B)--(C)--(D)--(A)--(E)--(F)--(G)--(H)--(E)--(D)--(H)--(A);
\draw (B)--(G)--(C)--(F)--(B);

\draw[fill] (A) circle (3pt);
\draw[fill] (B) circle (3pt);
\draw[fill] (C) circle (3pt);
\draw[fill] (D) circle (3pt);
\draw[fill] (E) circle (3pt);
\draw[fill] (F) circle (3pt);
\draw[fill] (G) circle (3pt);
\draw[fill] (H) circle (3pt);
\end{tikzpicture}
\caption{$G_3$.}\label{402}
\end{subfigure}

\caption{Three simple graphs whose co-degree sequence consists only of two distinct values.}
\end{figure}
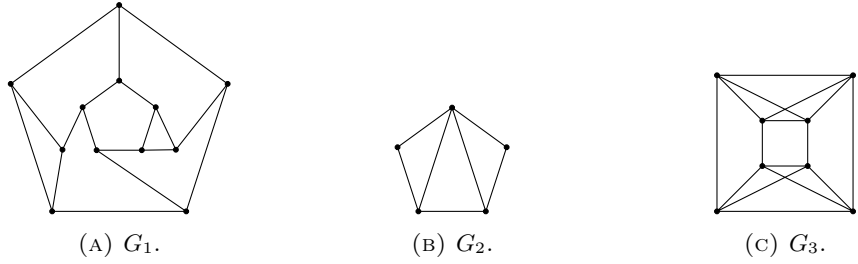

We also recall the following definition, see \cite{HKM}.
A $k$-regular graph on $v$ vertices is a \emph{divisible design graph} ($\DDG$ for short) with parameters $( v , k , \lambda , \mu , m , n )$ if the vertex set can be partitioned into $m$ classes of size $n$, such that two distinct vertices from the same class have exactly $\lambda$ common neighbors, and two vertices from different classes have exactly $\mu$ common neighbors.
Clearly, a $\DDG$ with parameters $(v,k, \lambda, \mu, m,n)$ is a Deza graph whose co-degree sequence is
    $$\left\{\begin{array}{ll}
    \left( \lambda^{m\binom{n}{2}}, \mu^{n^2 \binom{m}{2}} \right) & 
    \text{if }\lambda \geqslant \mu,\\
     \left( \mu^{n^2 \binom{m}{2}}, \lambda^{m\binom{n}{2}} \right) & 
    \text{if }\lambda < \mu.
    \end{array}\right.$$

The following result is immediate.
\begin{prop}
    Let $G$ be a $k$-regular graph on $v$ vertices such that $\ga(G)=\left(\lambda^\ell, \mu ^{\binom{v}{2}-\ell}\right)$, for some positive integer $\ell \le \binom{v}{2}$. Then $G$ is a Deza graph with parameters $(v,k,\lambda,\mu)$.

    Moreover, if the vertex set can be partitioned into $m$ classes of size $n$ such that two distinct vertices from the same class have exactly $\lambda$ common neighbors, and two vertices from different classes have exactly $\mu$ common neighbors, then $G$ is a divisible design graph.
\end{prop}

Note that the graph of Figure~\ref{221111} has co-degree sequence $(2^2,1^4)$, but it is not regular (so it is not a Deza graph). Still in the class of graphs whose co-degree sequence consists of at most two distinct values, we recall the following definition, given by Mulder.

\begin{defi}\cite{M}
    A connected loopless graph $G$ is called a \emph{$(0, \lambda)$-graph} if any two distinct vertices in $G$ have $\lambda \ge 2$ common neighbors or none at all.
\end{defi}

As shown in \cite[Proposition 1]{M}, a $(0, \lambda)$-graph is necessarily regular. Hence, the class of Deza graphs with parameters $(v, k, \lambda, 0)$, $\lambda \geqslant 2$, coincides with the class of $(0, \lambda)$-graphs. Figure~\ref{402} shows a $(0,2)$-graph which is $4$-regular and whose co-degree sequence is $(2^{24}, 0^4)$. 

We resume in the following proposition some of the results contained in \cite{M}.

\begin{prop}\cite[Propositions 4, 5, 6 and 8]{M}
Let $G$ be a $(0,\lambda)$-graph, for some $\lambda\geqslant 2$.
  \begin{itemize}
        \item[(1)] Suppose $\diam(G)=1$. Then $G$ is either $K_2$ or $K_{\lambda+2}$.
        \item[(2)] Suppose $\diam(G)=2$. If $G$ is bipartite or if $\Delta(G)=\lambda$, then $G$ is $K_{\lambda,\lambda}$.
        \item[(3)] Suppose $\diam(G)=2$. If $\lambda < \Delta(G) < 2\lambda$, then 
        $\ga(G)=\left(\lambda^{\binom{v}{2}}\right)$.
        \item[(4)] Suppose $\diam(G)=3$.
        If $\Delta(G) < 2\lambda$, then $G$ is bipartite.

    \end{itemize}
\end{prop}

In the following sections, we will consider (not necessarily connected) simple graphs whose co-degree sequence is binary, i.e., consists only of zeroes and ones. In this case, the regularity is not always guaranteed.

A similar problem has been considered by Maffucci in \cite{Maf}. Given a simple graph $G=(V,E)$, set
$$A(G)=\{a : \text{there exist } x,y \in V \text{ such that } c_{x,y}=a \}.$$
In particular, the problem posed in \cite{Maf} is to classify all polyhedra according to the set $A(G)$ or, in other words, according to the underlying set of the list $\ga(G)$.

\section{\texorpdfstring{$4$-cycle-free}{4-cycle-free} graphs}\label{free4}

In this section we focus our attention on simple graphs on $v$ vertices with the property that two vertices have at most one common neighbor; in other words, we are interested in those simple graphs  whose co-degree sequence $\ga$ is binary, i.e., has entries equal to $0$ or $1$ only:
$$\ga=\gamma_c:=\left(1^c, 0^{\binom{v}{2}-c}   \right)\quad \text{for some } c\in \left[0, \binom{v}{2} \right].$$
As already remarked in \cite{M}, a simple graph $G$ on $v$ vertices is such that $\ga(G)=\gamma_c$ for some $c\in \left[0, \binom{v}{2} \right]$ if and only if $G$ is $C_4$-free, meaning that no cycle of length four occurs.

Let $\mathcal{Q}_v$ be the set of all simple graphs, up to isomorphism, on $v\geqslant 2$ vertices which are $C_4$-free (see Figure~\ref{4} for the case $v=4$).

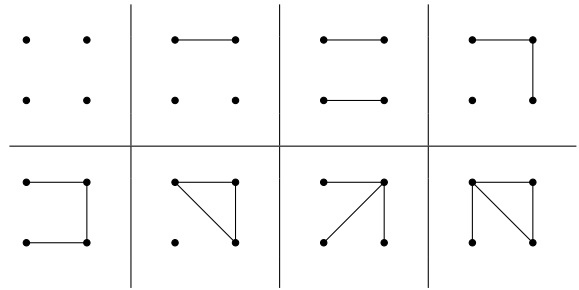
\begin{figure}[htbp]
\begin{center}
    \begin{tabular}{c|c|c|c}
    &&&\\
\begin{tikzpicture}[scale = 0.4]
\coordinate (A) at (0,2);
\coordinate (B) at (2,2);
\coordinate (C) at (2,0);
\coordinate (D) at (0,0);
\draw[fill] (A) circle (3pt);
\draw[fill] (B) circle (3pt);
\draw[fill] (C) circle (3pt);
\draw[fill] (D) circle (3pt);
\end{tikzpicture}\quad{}
 & 
\quad \begin{tikzpicture}[scale = 0.4]
\coordinate (A) at (0,2);
\coordinate (B) at (2,2);
\coordinate (C) at (2,0);
\coordinate (D) at (0,0);
\draw (A)--(B);
\draw[fill] (A) circle (3pt);
\draw[fill] (B) circle (3pt);
\draw[fill] (C) circle (3pt);
\draw[fill] (D) circle (3pt);
\end{tikzpicture}\quad{}  & 
\quad \begin{tikzpicture}[scale = 0.4]
\coordinate (A) at (0,2);
\coordinate (B) at (2,2);
\coordinate (C) at (2,0);
\coordinate (D) at (0,0);
\draw (A)--(B);
\draw (C)--(D);
\draw[fill] (A) circle (3pt);
\draw[fill] (B) circle (3pt);
\draw[fill] (C) circle (3pt);
\draw[fill] (D) circle (3pt);
\end{tikzpicture}\quad {} & 
\quad \begin{tikzpicture}[scale = 0.4]
\coordinate (A) at (0,2);
\coordinate (B) at (2,2);
\coordinate (C) at (2,0);
\coordinate (D) at (0,0);
\draw (A)--(B)--(C);
\draw[fill] (A) circle (3pt);
\draw[fill] (B) circle (3pt);
\draw[fill] (C) circle (3pt);
\draw[fill] (D) circle (3pt);
\end{tikzpicture}\quad {} \\ &&&\\ \hline &&&\\
\begin{tikzpicture}[scale = 0.4]
\coordinate (A) at (0,2);
\coordinate (B) at (2,2);
\coordinate (C) at (2,0);
\coordinate (D) at (0,0);
\draw (A)--(B)--(C)--(D);
\draw[fill] (A) circle (3pt);
\draw[fill] (B) circle (3pt);
\draw[fill] (C) circle (3pt);
\draw[fill] (D) circle (3pt);
\end{tikzpicture}\quad{}
 & 
\quad \begin{tikzpicture}[scale = 0.4]
\coordinate (A) at (0,2);
\coordinate (B) at (2,2);
\coordinate (C) at (2,0);
\coordinate (D) at (0,0);
\draw (A)--(B)--(C)--(A);
\draw[fill] (A) circle (3pt);
\draw[fill] (B) circle (3pt);
\draw[fill] (C) circle (3pt);
\draw[fill] (D) circle (3pt);
\end{tikzpicture}\quad{}  & 
\quad \begin{tikzpicture}[scale = 0.4]
\coordinate (A) at (0,2);
\coordinate (B) at (2,2);
\coordinate (C) at (2,0);
\coordinate (D) at (0,0);
\draw (A)--(B)--(C);
\draw (D)--(B);
\draw[fill] (A) circle (3pt);
\draw[fill] (B) circle (3pt);
\draw[fill] (C) circle (3pt);
\draw[fill] (D) circle (3pt);
\end{tikzpicture}\quad {} & 
\quad \begin{tikzpicture}[scale = 0.4]
\coordinate (A) at (0,2);
\coordinate (B) at (2,2);
\coordinate (C) at (2,0);
\coordinate (D) at (0,0);
\draw (A)--(B)--(C)--(A)--(D);
\draw[fill] (A) circle (3pt);
\draw[fill] (B) circle (3pt);
\draw[fill] (C) circle (3pt);
\draw[fill] (D) circle (3pt);
\end{tikzpicture}\quad {} \\ &&&\\ 
\end{tabular}
\end{center}
\caption{The eight elements of  $\mathcal{Q}_4$.}\label{4}
\end{figure}

Since the co-degree sequence of a graph $G \in \mathcal{Q}_v$ is
binary, there is a unique $c \in 
\left[0,\binom{v}{2}\right]$ such that $\ga(G)=\gamma_c$. So, we can define a function
$\c: \mathcal{Q}_v\to \left[0,\binom{v}{2}\right]$
in such a way that a graph $G \in \mathcal{Q}_v$ has co-degree sequence $\gamma_{\c(G)}$. 
In other words, $\c(G)$ is the number of $1$'s in the sequence $\ga(G)$.
Problem~\ref{problema} for this class of graphs asks if, for a given $c\in\left[0, \binom{v}{2} \right]$, there exists $G \in \mathcal{Q}_v$ such that $\ga(G)=\gamma_c$.
As already seen, Theorem~\ref{friendship} implies the solution of this problem for the case $c=\binom{v}{2}$.
Moving to $c = 0$, ambiguities arise. Indeed, any graph in which the length of each path is at most equal to $1$ is a solution.
We conclude that the solution to our problem is not unique in general, meaning that two (or more) non-isomorphic graphs may be the solution for the same instance.

The following result relates the degree sequence and the co-degree one of a graph.

\begin{lem}
If $G$ is a simple graph on $v$ vertices, $\bm{\pi}= (d_1 , \ldots, d_v)$ is its degree sequence and $\ga(G)=\left(\lambda_1,\ldots,\lambda_{\binom{v}{2}}\right)$,
 then
$$\sum_{j=1}^{\binom{v}{2}} \lambda_j  =\sum_{i=1}^v \binom{d_i}{2}.$$
\end{lem}

\begin{proof}
We proceed by double counting the number of paths of length $2$ in $G$, the so called \emph{cherries}. First of all, such cherries can be counted starting from their central vertex $v_i$, whose degree is $d_i$. The cherries centered at the vertex $v_i$ are then $\binom{d_i}{2}$, and so the total number of cherries is $\sum\limits_{i=1}^v \binom{d_i}{2}$.
On the other hand, every couple of distinct vertices $x,y$ forms the endpoints of exactly $c_{x,y}$ cherries. Summing the co-degree of each of the $\binom{v}{2}$ pairs of vertices, we get $\sum\limits_{j=1}^{\binom{v}{2}} \lambda_j  =\sum\limits_{i=1}^v \binom{d_i}{2}$.
\end{proof}

In particular, if a graph $G \in \mathcal{Q}_v$ has degree sequence $\bm{\pi}= (d_1 , \ldots, d_v)$, then
\begin{equation}\label{cd}
\c(G)=\sum_{i=1}^v \binom{d_i}{2}.
\end{equation}
Furthermore, if $e$ is the number of edges of $G$, then 
$$\c(G)=\sum_{i=1}^{v} \binom{d_i}{2}=\frac{1}{2}\left( \sum_{i=1}^{v} d_i^2-\sum_{i=1}^{v} d_i\right) =\frac{1}{2}\left( \sum_{i=1}^{v} d_i^2-2e\right) = \frac{1}{2}M_1(G) - e,$$
where $M_1(G)=\sum\limits_{i=1}^n d_i^2$ is the so called \emph{first Zagreb index} of $G$, see \cite{Zag}.
So, we can write
$$M_1(G)=2(\c(G)+e).$$

\begin{rem}\label{CS}
By applying the Cauchy-Schwarz inequality, it is easy to prove that $M_1(G)\geqslant \frac{4e^2}{v}$ for every simple graph $G$ with $v$ vertices and $e$ edges, where the equality holds if and only if $G$ is regular.
\end{rem}

\begin{lem}\label{G-x}
Let $G\in \mathcal{Q}_v$, $v\geqslant 2$,  and let $x\in V$. Then $G-x \in \mathcal{Q}_{v-1}$ and
\begin{equation}\label{eqG-x}
\c(G-x)=\c(G) - \binom{\deg_G(x)}{2} - \sum\limits_{y\in N_G(x)} \left(\deg_G(y)-1\right).
\end{equation}
\end{lem}

\begin{proof}
The graph $G-x$ is obtained from $G$ by removing the vertex $x$ and all edges incident on it, so $G-x$ is still $C_4$-free. Moreover, $\c(G-x)$ is computed by subtracting from $\c(G)$ the number of pairs of neighbors of $x$, which lose a common neighbor (namely, $\binom{\deg_G(x)}{2}$), and the number of cherries having $x$ as an endpoint and a vertex $y\in N_G(x)$ as central vertex ($\deg_G(y)-1$ instances for each $y\in N_G(x)$). This proves the validity of \eqref{eqG-x}.
\end{proof}

\begin{prop}\label{v-2}
Let $G \in \mathcal{Q}_v$ with $v\geqslant 3$ and
suppose that  $G$ is connected. Then $\c(G)\geqslant v-2$,
where the equality holds if and only if $G$ is the path on $v$ vertices.
\end{prop}

\begin{proof}
If $G$ is $k$-regular, then $k\geqslant 2$ and $\c(G)=v \binom{k}{2}\geqslant v$  by \eqref{cd}. So, assume that $G$ is not regular. Then,  $M_1(G) > \frac{4e^2}{v}$
by Remark~\ref{CS}.
Since $G$ is connected, we have $e\geqslant v-1$, whence
$$\c(G)> e \left(\frac{2e}{v}-1\right) \geqslant \frac{(v-1)(v-2)}{v}\geqslant v-3.$$

It is easy to see that a path $P$ on $v$ vertices is such that
$\c(P) =v-2$; we prove, by induction on $v$, that a connected graph $G\in \mathcal{Q}_v$  with $\c(G ) = v - 2$ is necessarily a path.

The only connected graphs on $v=3$ vertices are the path, for which 
$c=v-2=1$, and the complete graph $K_3$, for which $c=3$.
So, let us consider a connected graph 
$G\in \mathcal{Q}_v$, $v\geqslant 4$, having $\c(G)=v-2$. 
Note that $G$ has a vertex $x$ of degree $1$, otherwise
$\c(G)\geqslant v \binom{2}{2}> v-2$, a contradiction.
By Lemma~\ref{G-x}, we have
$$\c(G-x)=\c(G)-\deg_G(y)+1=v-1-\deg_G(y).$$
Since $G-x\in \mathcal{Q}_{v-1}$ is still connected, we obtain that $\c(G-x)\geqslant v-3$ by what we proved before.
It follows that $\deg_G(y)\leqslant 2$ and hence $\deg_G(y)=2$, being $G-x$ connected. In particular, $\c(G-x)=v-3$ and, by the induction hypothesis, $G-x$ is the path on $v-1$ vertices. We conclude that  $G$ is a path.
\end{proof}

\section{Trees}\label{alberi}

A tree $T$ is a connected acyclic graph. Among many properties of trees, we recall that any two vertices of $T$ can be connected by a unique path and that, if $T$ is defined on $v$ vertices, then it has exactly $v-1$ edges.
We also recall that, in a tree, every two vertices have at most one common neighbor.
For instance, the graph in Figure~\ref{star} is a tree, the so called \emph{star graph}.

\begin{figure}[ht]
\begin{tikzpicture}[scale = 0.3]

\draw 
  node (p) [
    minimum size=3cm,
    regular polygon, 
    regular polygon sides=10
  ] {};
  
\foreach \n in {1,2,...,10}{
 \node[anchor=\n*(360/10)] at (p.corner \n) {};
}
\coordinate (A) at (p.corner 1);
\coordinate (B) at (p.corner 2);
\coordinate (C) at (p.corner 3);
\coordinate (D) at (p.corner 4);
\coordinate (E) at (p.corner 5);
\coordinate (F) at (p.corner 6);
\coordinate (G) at (p.corner 7);
\coordinate (H) at (p.corner 8);
\coordinate (I) at (p.corner 9);
\coordinate (J) at (p.corner 10);
\coordinate (O) at (0,0);

\draw (B)--(O)--(A);
\draw (D)--(O)--(C);
\draw (F)--(O)--(E);
\draw (G)--(O)--(H);
\draw (J)--(O)--(I);

\draw[fill] (A) circle (3pt);
\draw[fill] (B) circle (3pt);
\draw[fill] (C) circle (3pt);
\draw[fill] (D) circle (3pt);
\draw[fill] (E) circle (3pt);
\draw[fill] (F) circle (3pt);
\draw[fill] (G) circle (3pt);
\draw[fill] (H) circle (3pt);
\draw[fill] (I) circle (3pt);
\draw[fill] (J) circle (3pt);
\draw[fill] (O) circle (3pt);

\end{tikzpicture}
\caption{The star graph on $v=11$ vertices.}\label{star}
\end{figure}
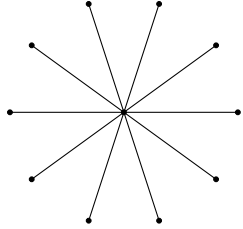

Given a tree $T$, we designate a vertex $r \in V$ as the \emph{root} of $T$. 
Given a vertex $x$, its \emph{parent} is the vertex $y$ adjacent to $x$ on the path to the root (a root has no parent), while $x$ is said to be the \emph{child} of $y$.
Any vertex has a unique parent, but can have any number of children, in general. The only exception is the root, that does not have any parent. 
The vertices of degree $1$ and different from the root are called \emph{leaves}; clearly, a leaf does not have any child. 
Finally, we call \emph{siblings} two (or more) vertices sharing the same parent.

The \emph{height} $\ell$ of a tree is the length (i.e., the number of edges) of the longest path connecting the root to a leaf. The \emph{level} $i$ is defined as the set of vertices at distance $i$ from the root. By convention, the root is at level $0$. We denote by $b_i$ the number of vertices that belong to level $i$ or, equivalently, the total number of children of all the vertices at level $i - 1$. 
We will represent a tree $T$ by drawing the root at the top and the siblings at level $i$ on the same horizontal line, as in Figure~\ref{F3.13}. Such a representation allows us to speak of the leftmost vertex at level $i$.

We denote by $T (v, c)$ an instance of the problem of finding a tree on $v$ vertices having co-degree sequence equal to  $\gamma_c$.

\begin{thm}\label{Mic}
If $\gamma_c$ is the co-degree sequence of a tree on $v$ vertices, then
$v-2\leqslant c \leqslant \binom{v-1}{2}$. 
Moreover, when choosing $c=v-2$ or $c=\binom{v-1}{2}$, the solution of $T(v,c)$ is unique, and consists in the first case of the path, in the latter of the star graph.
\end{thm}

\begin{proof}
As a tree is connected, the minimum value of $c$ is $c_{\min}= v-2$ by Proposition~\ref{v-2}; in this case, the tree is necessarily  a path. 
Also, since in trees a path connecting two vertices is unique, counting the number of vertices sharing exactly one neighbor is equivalent to count the number of pairs at distance $2$. 
A tree on $v$ vertices has exactly $v-1$ edges, so there are exactly $v-1$ pairs of vertices at distance $1$. 
As a consequence, the number of pairs at distance at least  $2$ is equal to $\binom{v}{2}-(v-1)=\binom{v-1}{2}$. So, 
$$c=|\{\text{pairs at distance } 2\}| =
\binom{v-1}{2} - |\{\text{pairs at distance at least }  3\}|.$$
It follows that the maximum admitted value for $c$ is reached when no pairs of vertices are at distance greater than or equal to $3$, that is, $c_{\max} = \binom{v-1}{2}$. 
The only tree realizing such a situation is the star, see Figure~\ref{star}.
\end{proof}

\subsection{Canonical trees}

Starting from our observations about the ambiguity of reconstruction, we define a canonical structure for trees, from which we can deduce an easy formula for the computation of their co-degree sequence.
To do that, we define the operation of \emph{left-shift}.

\begin{defi}\label{3.5.4}
Given a tree $T$, let $\nu_y$ be the leftmost  leaf in $T$, and let us consider a vertex $\nu_x$ at level $x \geqslant 1$ and not on the path from $\nu_y$ to the root. Let $L(x)$ denote the set of descendants of $\nu_x$. The \emph{left-shift operation} applied on $\nu_x$ consists in moving $L(x)$ such that $\nu_y$ becomes its root and $\nu_x$ becomes a leaf. 
\end{defi}

The tree $\tilde T$  of Figure~\ref{F3.13b} has been obtained from the tree $T$ of Figure~\ref{F3.13}
by applying a left-shift operation.

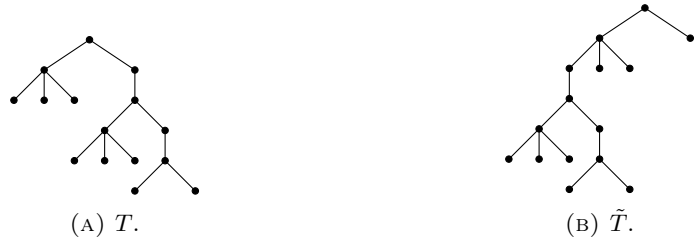
\begin{figure}[ht]
\centering
\begin{subfigure}{0.45\textwidth}
\centering
\begin{tikzpicture}[scale = 0.4]

\coordinate (A) at (0,0);
\coordinate (B) at (1,0);
\coordinate (C) at (2,0);
\coordinate (D) at (1,1);
\coordinate (E) at (2.5,2);
\coordinate (F) at (4,1);
\coordinate (G) at (4,0);
\coordinate (H) at (3,-1);
\coordinate (I) at (5,-1);
\coordinate (J) at (2,-2);
\coordinate (K) at (3,-2);
\coordinate (L) at (4,-2);
\coordinate (M) at (5,-2);
\coordinate (N) at (4,-3);
\coordinate (O) at (6,-3);

\draw (A)--(D)--(B);
\draw (C)--(D)--(E)--(F)--(G)--(H)--(J);
\draw (K)--(H)--(L);
\draw (G)--(I)--(M)--(N);
\draw (M)--(O);
\draw[fill] (A) circle (3pt);
\draw[fill] (B) circle (3pt);
\draw[fill] (C) circle (3pt);
\draw[fill] (D) circle (3pt);
\draw[fill] (E) circle (3pt);
\draw[fill] (F) circle (3pt);
\draw[fill] (G) circle (3pt);
\draw[fill] (H) circle (3pt);
\draw[fill] (I) circle (3pt);
\draw[fill] (J) circle (3pt);
\draw[fill] (K) circle (3pt);
\draw[fill] (L) circle (3pt);
\draw[fill] (M) circle (3pt);
\draw[fill] (N) circle (3pt);
\draw[fill] (O) circle (3pt);
\end{tikzpicture}
\caption{$T$.}\label{F3.13}
\end{subfigure}
\begin{subfigure}{0.45\textwidth}
\centering
\begin{tikzpicture}[scale = 0.4]

\coordinate (A) at (0,0);
\coordinate (B) at (1,0);
\coordinate (C) at (2,0);
\coordinate (D) at (1,1);
\coordinate (E) at (2.5,2);
\coordinate (F) at (4,1);

\coordinate (G) at (0,-1);
\coordinate (H) at (-1,-2);
\coordinate (I) at (1,-2);

\coordinate (J) at (-2,-3);
\coordinate (K) at (-1,-3);
\coordinate (L) at (0,-3);
\coordinate (M) at (1,-3);
\coordinate (N) at (0,-4);
\coordinate (O) at (2,-4);

\draw (A)--(D)--(B);
\draw (C)--(D)--(E)--(F);

\draw (A)--(G)--(H)--(J);
\draw (K)--(H)--(L);
\draw (G)--(I)--(M)--(N);
\draw (M)--(O);
\draw[fill] (A) circle (3pt);
\draw[fill] (B) circle (3pt);
\draw[fill] (C) circle (3pt);
\draw[fill] (D) circle (3pt);
\draw[fill] (E) circle (3pt);
\draw[fill] (F) circle (3pt);
\draw[fill] (G) circle (3pt);
\draw[fill] (H) circle (3pt);
\draw[fill] (I) circle (3pt);
\draw[fill] (J) circle (3pt);
\draw[fill] (K) circle (3pt);
\draw[fill] (L) circle (3pt);
\draw[fill] (M) circle (3pt);
\draw[fill] (N) circle (3pt);
\draw[fill] (O) circle (3pt);

\end{tikzpicture}
\caption{$\tilde T$.}\label{F3.13b}
\end{subfigure}
\caption{Two trees on $15$ vertices with co-degree sequence $\gamma_{21}$.}
\end{figure}

\begin{thm}\label{3.5.5}
Given a tree $T$ with co-degree sequence $\gamma_c$, let $T'$ be obtained from $T$ after the application of the left-shift operation. Then, 
$\ga(T')=\gamma_c$, i.e., the left-shift operation does not change the co-degree sequence of a tree.
\end{thm}

\begin{proof}
Let  $\nu_x$ the vertex on which we apply the left-shift operation to obtain $T'$ from $T$. 
Let us consider $T_x = L(x) \cup \{\nu_x\}$, the sub-tree composed by $\nu_x$, its root, and all its descendants, and let us denote $c_x$ the number of pairs of vertices in $T_x$ sharing one neighbor. Moreover, let $d$ be the number of children of $\nu_x$. Since $\nu_x$ is at level $x\geqslant 1$ in $T$, by hypothesis, $d$ is also the number of vertices in $T_x$ that share the neighbor $\nu_x$ with some vertex in $T \setminus  T_x$, precisely with the parent of $\nu_x$. As a consequence, $c = c_x + d + d'$, where $d'$ is the number of pairs of vertices in $T \setminus T_x$ sharing one neighbor.

Let us now consider the left-shift operation applied on $\nu_x$, consisting in moving $L(x)$ such that the leftmost leaf of $T$, say $\nu_y$, becomes its root. 
We underline that $\nu_y \not \in L(x)$, since $\nu_x$ is not on the path from 
$\nu_y$ to the root by hypothesis. 
We denote by $T'$ the resulting tree. Then, being $\gamma_{c'}$ its co-degree sequence, the value $c'$ can be computed following the same argument used for $c$, just replacing $\nu_x$ with $\nu_y$. Notice that $\nu_y$ is at level $y \geqslant 1$ as well, since it was the leftmost leaf in $T$, and that the number of pairs counted by 
$d$ and $d'$, previously defined, does not change. As a consequence, 
$c = c'$, and the statement follows.
\end{proof}

We say that a tree is \emph{canonical}, if it is such that the
left-shift operation cannot be further applied (no suitable vertex is left).

\begin{rem}\label{3.5.6}
If $T$ is canonical, then all its levels are composed by siblings, i.e., all the vertices on the same level have the same parent.
\end{rem}

From Theorem~\ref{3.5.5}, we deduce that if a solution to $T(v, c)$ exists, then a solution to the same instance, and having a canonical structure, exists too. So, from now on, we consider canonical trees only.
As a consequence, we can reformulate the co-degree reconstruction
problem as follows.

\begin{prob}\label{4.5}
Given a binary admissible sequence $\gamma_c$ of length $\binom{v}{2}$, for some $v \geqslant 2$, determine
whether there exists a \emph{canonical} tree on $v$ vertices having 
$\gamma_c$ as its co-degree sequence.
\end{prob}

Due to the structure of canonical trees, it is possible to find a closed formula to compute the value $c$ that defines the binary sequence $\gamma_c$.
By Remark~\ref{3.5.6}, a canonical tree is characterized by a path, on the left, whose length coincides with the height $\ell$ of the tree. 
We denote the vertices on this path, from the root to the leaf, as $r_0 , \ldots, r_\ell$. As a direct consequence of Remark~\ref{3.5.6}, each $r_i$, for
$i = 0, \ldots , \ell - 1$, is exactly the parent of $b_{i+1}$ vertices, corresponding to the siblings that constitute the level $i + 1$ of the tree.

\begin{thm}\label{3.5.8}
Let $T$ be a canonical tree on $v$ vertices of height $\ell$, with $b_0=1, b_1 , \ldots, b_\ell$ the cardinalities of its levels.  Then
\begin{equation}\label{3.10}
\c(T)=\sum_{i=1}^\ell \binom{b_i}{2} +v-1-b_1.
\end{equation}
\end{thm}

\begin{proof}
The number of vertices in a tree is given by the sum of the cardinalities of its levels: $v=\sum\limits_{i=0}^\ell b_i$.
By convention, the level of the root has cardinality 
$b_0 = 1$, so that
\begin{equation}\label{3.7}
\sum_{i=2}^\ell b_i =  v-b_1 -1.
\end{equation}
We count the number of pairs of vertices sharing one neighbor. We start considering
siblings: since $T$ is canonical, each level is composed by exactly $b_i$ siblings, for $i = 1, \ldots, \ell$.
By definition, any pair of siblings at level $i$ has the parent 
$r_{i-1}$ as common neighbor, for
$i = 1, \ldots,\ell$. 
Then, when considering siblings, each level contributes to $c$ for a value 
\begin{equation}\label{3.8}
c_i = \binom{b_i}{2},\quad \text{for } i = 1, \ldots,\ell.
\end{equation}
Furthermore, on each level a vertex shares one neighbor, specifically its parent $r_{i-1}$, with one vertex lying on the longest path, specifically $r_{i-2}$ (that is, the parent of its parent).
Notice that level $1$ is excluded from this computation, since the root, at level zero, has no parent by definition. 
The contribution to $c=\c(T)$ given by this type of pairs is expressed by
the number of the involved vertices, that is,
\begin{equation}\label{3.9}
c^*=\sum_{i=2}^\ell b_i.
\end{equation}
To complete our analysis, we have to consider a last type of pair: two vertices at different levels and not lying on the longest path, say $ x \in b_i$ and $y \in b_j$, with 
$1 \leqslant  i < j \leqslant \ell$. 
It is clear from the structure of canonical trees that the distance between $x$ and $y$ is $j - i + 2$,
so greater than or equal to $3$ for each $i \neq j$. 
So, this type of pair is not involved in the computation of $c$.

We conclude that $c$ is obtained as the sum of \eqref{3.8}, for 
$i = 1, \ldots, \ell$, and \eqref{3.9}:
$$c=\sum_{i=1}^\ell \binom{b_i}{2}+\sum_{i=2}^\ell b_i.$$
Replacing \eqref{3.7} in the previous equation gives
$c =\sum\limits_{i=1}^\ell \binom{b_i}{2}+ v-1- b_1$, as requested.
\end{proof}

From Theorem~\ref{3.5.8}, we reformulate the reconstruction problem in an arithmetical flavor.
Indeed, given two integers $v$ and $c$, we have to find $\ell$ suitable integers $b_1 , \ldots,b_\ell$ such that
$$c =\sum_{i=1}^\ell \binom{b_i}{2}+v-1-b_1.$$
We underline again that many solutions may exist.
This is reflected in the possibility of different choices for the parameter $\ell$, as well as for the parameter $b_1$, in
general.

\subsection{Levels of canonical trees and their cardinalities}

Let us consider a canonical tree $T$ of height $\ell$, and its co-degree sequence $\gamma_c$. We focus our attention on the cardinalities of its levels, $b_1 , \ldots, b_\ell$, that lead to a quick computation of the parameter $c$.  It is clear from \eqref{3.10} that $b_1$ plays a special role in the computation of $c$, while $b_2 ,\ldots,b_\ell$ are equally significant. From this observation, we choose an order on the cardinalities of the levels of a canonical tree, in order to reach a well-defined structure.

Let us consider two indices $2 \leqslant i < j \leqslant \ell$. Given a canonical tree $T$ of 
height $\ell$ and with co-degree sequence $\gamma_c$, if we exchange the vertices at level $i$ with the vertices at level $j$, we get a tree $T'$ that is still canonical, and with the same co-degree sequence. This last property directly follows from \eqref{3.10}.

As a consequence, we define a further operation on canonical trees, that consists of the
following: if $b_1 = 1$, we do not change the tree. If $b_1 \geqslant 2$, we designate one of the children of the root, different from $r_1$, as the new root of the tree. Actually, this operation does not
change the structure of the graph, since it consists in the renaming of the root only. On
the other hand, from a practical point of view, it allows to get $b_1 = 1$ for any canonical
tree. We also point out that the canonical tree thus obtained has height increased by one
with respect to the previous one. Example~\ref{3.5.9} clarifies the described operation.

Finally, we re-order the levels of the tree, except for the first one, in decreasing order
with respect to their cardinalities, that is, $b_2 \geqslant b_3 \geqslant \cdots \geqslant 
b_\ell$.
Equation \eqref{3.10}  ensures that such a re-ordering does not affect the co-degree sequence $\gamma_c$.

\begin{ex}\label{3.5.9}
Let us consider the tree $T$ in Figure~\ref{F3.13}, whose co-degree sequence is
$\ga= \gamma_{21}=(1^{21} , 0^{84} )$. 
Figure~\ref{3.14a} shows a second tree, $T'$, whose height is $\ell' = 7$ and with
same co-degree sequence as $T$, obtained from $T$ after the iterative application of the left-shift operation. It is canonical.
In Figure~\ref{3.14b}, the operation that leads to $b_1 = 1$ is shown. 
Notice that $T''$ has height increased by one with respect to $T'$: $\ell'' = 8$.
Finally, Figure~\ref{3.14c} represents the canonical tree $T'''$, that is obtained from $T''$ after the re-ordering of its levels with respect to the cardinalities.
\end{ex}

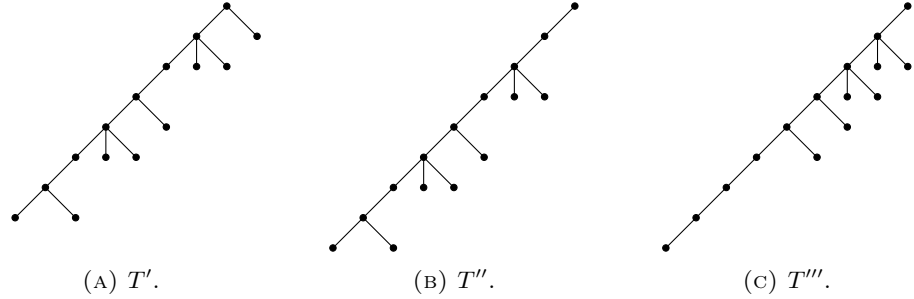
\begin{figure}[ht]
\centering
\begin{subfigure}{0.3\textwidth}
\centering
\begin{tikzpicture}[scale = 0.4]

\coordinate (A) at (0,0);
\coordinate (B) at (1,0);
\coordinate (C) at (2,0);
\coordinate (D) at (1,1);
\coordinate (E) at (2,2);
\coordinate (F) at (3,1);
\coordinate (G) at (-1,-1);
\coordinate (H) at (-2,-2);
\coordinate (I) at (-3,-3);
\coordinate (J) at (-4,-4);
\coordinate (K) at (-5,-5);

\coordinate (L) at (0,-2);
\coordinate (M) at (-2,-3);
\coordinate (N) at (-1,-3);
\coordinate (O) at (-3,-5);
\coordinate (Z) at (-6,-6);

\draw (K)--(J)--(I)--(H)--(G)--(A)--(D)--(B);
\draw (C)--(D)--(E)--(F);
\draw (G)--(L);
\draw (M)--(H)--(N);
\draw (J)--(O);
\draw[fill] (A) circle (3pt);
\draw[fill] (B) circle (3pt);
\draw[fill] (C) circle (3pt);
\draw[fill] (D) circle (3pt);
\draw[fill] (E) circle (3pt);
\draw[fill] (F) circle (3pt);
\draw[fill] (G) circle (3pt);
\draw[fill] (H) circle (3pt);
\draw[fill] (I) circle (3pt);
\draw[fill] (J) circle (3pt);
\draw[fill] (K) circle (3pt);
\draw[fill] (L) circle (3pt);
\draw[fill] (M) circle (3pt);
\draw[fill] (N) circle (3pt);
\draw[fill] (O) circle (3pt);
\draw[white] (Z) circle (3pt);

\end{tikzpicture}
\caption{$T'$.}\label{3.14a}
\end{subfigure}
\begin{subfigure}{0.3\textwidth}
\centering

\begin{tikzpicture}[scale = 0.4]
\centering
\coordinate (A) at (0,0);
\coordinate (C) at (-4,-6);
\coordinate (D) at (1,1);
\coordinate (E) at (-3,-4);
\coordinate (F) at (-6,-6);
\coordinate (G) at (-1,-1);
\coordinate (H) at (-2,-2);
\coordinate (I) at (-3,-3);
\coordinate (J) at (-4,-4);
\coordinate (K) at (-5,-5);
\coordinate (L) at (-1,-3);
\coordinate (M) at (0,-1);
\coordinate (N) at (1,-1);
\coordinate (O) at (-2,-4);

\coordinate (B) at (2,2);

\draw (F)--(K)--(J)--(I)--(H)--(G)--(A)--(D)--(B);
\draw (H)--(L);
\draw (M)--(A)--(N);
\draw (E)--(I)--(O);
\draw (K)--(C);
\draw[fill] (A) circle (3pt);
\draw[fill] (B) circle (3pt);
\draw[fill] (C) circle (3pt);
\draw[fill] (D) circle (3pt);
\draw[fill] (E) circle (3pt);
\draw[fill] (F) circle (3pt);
\draw[fill] (G) circle (3pt);
\draw[fill] (H) circle (3pt);
\draw[fill] (I) circle (3pt);
\draw[fill] (J) circle (3pt);
\draw[fill] (K) circle (3pt);
\draw[fill] (L) circle (3pt);
\draw[fill] (M) circle (3pt);
\draw[fill] (N) circle (3pt);
\draw[fill] (O) circle (3pt);
\end{tikzpicture}
\caption{$T''$.}\label{3.14b}
\end{subfigure}
\begin{subfigure}{0.3\textwidth}
\centering

\begin{tikzpicture}[scale = 0.4]
\coordinate (A) at (0,0);
\coordinate (C) at (-1,-3);
\coordinate (D) at (1,1);
\coordinate (E) at (1,-1);
\coordinate (F) at (-6,-6);
\coordinate (G) at (-1,-1);
\coordinate (H) at (-2,-2);
\coordinate (I) at (-3,-3);
\coordinate (J) at (-4,-4);
\coordinate (K) at (-5,-5);
\coordinate (L) at (0,-2);
\coordinate (M) at (2,0);
\coordinate (N) at (1,0);
\coordinate (O) at (0,-1);
\coordinate (B) at (2,2);

\draw (F)--(K)--(J)--(I)--(H)--(G)--(A)--(D)--(B);
\draw (G)--(L);
\draw (M)--(D)--(N);

\draw (E)--(A)--(O);
\draw (H)--(C);

\draw[fill] (A) circle (3pt);
\draw[fill] (B) circle (3pt);
\draw[fill] (C) circle (3pt);
\draw[fill] (D) circle (3pt);
\draw[fill] (E) circle (3pt);
\draw[fill] (F) circle (3pt);
\draw[fill] (G) circle (3pt);
\draw[fill] (H) circle (3pt);
\draw[fill] (I) circle (3pt);
\draw[fill] (J) circle (3pt);
\draw[fill] (K) circle (3pt);
\draw[fill] (L) circle (3pt);
\draw[fill] (M) circle (3pt);
\draw[fill] (N) circle (3pt);
\draw[fill] (O) circle (3pt);
\end{tikzpicture}
\caption{$T'''$.}\label{3.14c}
\end{subfigure}
\caption{Three canonical trees on $15$ vertices with co-degree sequence $\gamma_{21}$. In \ref{3.14b}
and \ref{3.14c}, level $1$ has cardinality equal to $b_1 = 1$. In \ref{3.14c}, the levels are in decreasing order with respect to their cardinalities.}\label{3.14}
\end{figure}

\begin{lem}\label{3.5.10}
Given a binary sequence $\gamma_c$, if there exists a tree having $\gamma_c$ as its 
co-degree sequence, then there exists a canonical tree realizing $\gamma_c$ whose levels
satisfy:
\begin{itemize}
\item[$(1)$] $b_1 = 1$;
\item[$(2)$] $b_2 \geqslant b_3 \geqslant \cdots\geqslant  
b_\ell$, with $\ell$ being  the height of the tree.
\end{itemize}
\end{lem}

\begin{proof}
The statement is a direct consequence of the operations we described above.
\end{proof}

We conclude that the reconstruction problem 
consists in finding $b_2 \geqslant b_3 \geqslant \ldots\geqslant  b_\ell$ such that
$$c=\sum_{i=2}^\ell \binom{b_i}{2} +v-2.
$$
Notice that $\ell$ can still vary, and more solutions may exist, in general.

\subsection{Trees and partitions}

Let $T$ be a canonical tree of $v\geqslant 3$ vertices of height $\ell$ with 
$b_0=1,b_1,\ldots, b_\ell$ the cardinalities of its levels. 
By Lemma~\ref{3.5.10} we may assume $b_1=1$ and 
$b_2 \geqslant b_3 \geqslant \cdots\geqslant  
b_\ell$. So, the sequence $(b_2,b_3,\ldots,b_\ell)$ is a partition of $v-2$.
We can exploit some properties of the partitions of an integer to provide an answer to Problem~\ref{4.5}.

Given a partition $\alpha=(\alpha_1,\alpha_2,\ldots,\alpha_k)$ 
of a positive integer $m$, define
$$\bm{c}(\alpha)=m+\sum_{i=1}^k \binom{\alpha_i}{2}.$$

Given two partitions $\alpha=(\alpha_1,\alpha_2,\ldots,\alpha_k)$
and $\beta=(\beta_1,\beta_2,\ldots,\beta_h$) of a positive integer $m$, consider the \emph{dominance} order as follows. Set
$\alpha_{k+1}=\ldots=\alpha_m=0$ and $\beta_{h+1}=\ldots=\beta_m=0$.
Write $\alpha \unlhd\beta$ if
$$\sum_{i=1}^t \alpha_i \leqslant \sum_{i=1}^t \beta_i
\quad \text{ for all } 1\leqslant t \leqslant m;$$
write $\alpha \lhd\beta$ if  $\alpha \unlhd\beta$ and $\alpha \neq \beta$.
Note that the dominance order is a partial order relation on the set of all partitions of $m$. We say that two partitions $\alpha,\beta$ of $m$ are \emph{comparable} if either $\alpha \unlhd \beta$ or $\beta \unlhd \alpha$.

To any  partition $\alpha=(\alpha_1,\alpha_2,\ldots,\alpha_k)$ of a positive integer $m$, we can associate the partition $\alpha'=(\alpha_1',\alpha_2',\ldots,\alpha_{\alpha_1}')$ of $m$, where $\alpha_i' = |\{\alpha_j\geqslant i \}|$ for all $i=1,\ldots,\alpha_1$. 
For instance, $(5,4,4,2,1,1)' = (6,4,3,3,1)$.

\begin{lem}\cite[Lemma 1.4.11]{JK}\label{conj}
    Let $\alpha,\beta$ be two partitions of a positive integer $m$.
    If $\alpha \lhd \beta$, then $\beta'\lhd \alpha'$.
\end{lem}

Given two partitions $\alpha,\beta$ of a positive integer $m$, write $\alpha \prec \beta$ if $\alpha \lhd \beta$ and there is no partition $\rho$ of $m$ such that $\alpha \lhd \rho \lhd \beta$. For the following result, see \cite[Proposition 2.3]{Br} or \cite[Theorem 1.4.10]{JK}.

\begin{lem}\label{domstretta}
Let $\alpha=(\alpha_1 , \ldots, \alpha_m)$ and 
$\beta = (\beta_1 , \ldots, \beta_m)$ be partitions of a positive integer $m$, allowing the last parts to be zero.
Then $\alpha \prec\beta$ if and only if  at least one of the following cases occurs:
\begin{itemize}
    \item[$(1)$] There exists $j \in [1,m]$ such that
    $\beta_j = \alpha_j+1$, $\beta_{j+1} = \alpha_{j+1} - 1$ and $\alpha_i=\beta_i$ for all 
    $i\not \in \{j,j+1\}$.
\item[$(2)$] There exist $j,k \in [1,m]$ such that $j<k$, $\beta_j=\alpha_k+1$,
$\beta_k=\alpha_k-1$, $\alpha_j=\alpha_k$ and $\alpha_i=\beta_i$ for all
$i \not \in \{j,k\}$.
\end{itemize}
\end{lem}

\begin{lem}\label{part}
Let $\alpha,\beta$ be two partitions of a positive integer $m$. 
If $\alpha \lhd \beta$ then $\bm{c}(\alpha)< \bm{c}(\beta)$. Moreover, if $\alpha$ and $\beta$ are comparable and
$\bm{c}(\alpha)< \bm{c}(\beta)$, then $\alpha \lhd \beta$. 
\end{lem}

\begin{proof}
The first part of the statement follows once we prove it for $\alpha \prec\beta$. By Lemma~\ref{domstretta} we have two cases. 
    
If case $(1)$ of Lemma~\ref{domstretta} holds, then there exists an index $j \in [1,m]$ such that $\beta_j = \alpha_j+1$, $\beta_{j+1} = \alpha_{j+1} - 1$ and $\alpha_i=\beta_i$ for all $i\not \in \{j,j+1\}$. Then $\bm{c}(\beta) - \bm{c}(\alpha) = \binom{\alpha_j + 1}{2} + \binom{\alpha_{j+1} - 1}{2} - \binom{\alpha_j}{2} - \binom{\alpha_{j+1}}{2} = \alpha_j - \alpha_{j+1} + 1$. Since $\alpha_j \geqslant \alpha_{j+1}$, we obtain that $\bm{c}(\beta) - \bm{c}(\alpha) > 0$.

If case $(2)$ of the lemma holds, then there exist $j,k $ such that $1\leqslant j<k\leqslant m$, $\beta_j=\alpha_k+1$, $\beta_k=\alpha_k-1$, $\alpha_j=\alpha_k$ and $\alpha_i=\beta_i$ for all $i \not \in \{j,k\}$. Then $\bm{c}(\beta) - \bm{c}(\alpha) = \binom{\alpha_k+1}{2}+\binom{\alpha_k-1}{2}-2 \binom{\alpha_k}{2} =1>0$.

Now, suppose that $\alpha$ and $\beta$ are two comparable partitions of $m$ such that $\bm{c}(\alpha)< \bm{c}(\beta)$. 
By the above, if $\beta \unlhd \alpha$ then $\bm{c}(\beta) \leqslant \bm{c}(\alpha)$, a contradiction. Hence, $\alpha\lhd \beta$.   
\end{proof}

Note that, taking  for instance the two partitions $\alpha=(4,1,1,1,1)$ and $\beta=(3,3,2)$ of $m=8$, we have $\bm{c}(\alpha)=14 < \bm{c}(\beta)=15$, but the two partitions are not comparable.
    
Going back to co-degree sequences and to Problem~\ref{4.5}, by Lemma~\ref{3.5.10} a solution to $T(v,c)$ exists if and only if there exists a partition $\alpha=(\alpha_1,\ldots,\alpha_k)$ of $v-2$ such that $\bm{c}(\alpha)=c$.
If so, one can take as solution the canonical tree $T_\alpha$ on $v$ vertices of height $k+1$, with $b_0=1,b_1=1,b_2=\alpha_1,\ldots,b_{k+1}=\alpha_k$ the cardinalities of its levels.
Note that $\c(T_\alpha)=\bm{c}(\alpha)$.

There are several algorithms for constructing all the partitions of a given integer, see \cite{Mi,ZS}. Using \textsc{gap}, we compute for each $v \leqslant 100$ all the values of $c \in \left[v-2,\binom{v-1}{2} \right]$ such that $\gamma_c$ is not the co-degree sequence of a tree on $v$ vertices, see Table~\ref{tab:c} for a partial list. The code is as follows.
\begin{verbatim}
for v in [3..100] do
 P:=Partitions(v-2);; 
 ok:=Set(P, p-> v-2 + Sum(p,k->Binomial(k,2)));
 d:=Difference([v-1..Binomial(v-1,2)],ok);
 Print("v = ",v,": Missing values of c = ",d,"\n");
od;
\end{verbatim}

For a reconstruction algorithm for an admissible value of $c$, see \cite[Section~3.5.3]{Asc}.

\begin{table}[ht]
$$\begin{array}{|c|l|}\hline
v & \text{Missing values of } c\\ \hline
 3 & -\\
 4 & - \\
 5 &  5 \\
 6 &  8, 9 \\
 7 &  10, 12, 13, 14 \\
 8 &  11, 14, 15, 17, 18, 19, 20 \\
 9 &  15, 19, 20, 21, 23, 24, 25, 26, 27 \\
 10 &  22, 25, 26, 27, 28, 30, 31, 32, 33, 34, 35 \\
 11 &  23, 26, 28, 29, 32, 33, 34, 35, 36, 38, 39, 40, 41, 42, 43, 44 \\
 12 &  29, 33, 35, 36, 37, 40, 41, 42, 43, 44, 45, 47, 48, 49, 50, 51,  52, 53, 54 \\
 13 &  37, 41, 43, 44, 45, 46, 49, 50, 51, 52, 53, 54, 55, 57, 58, 59,      60, 61, 62, 63, 64, 65 \\
14 &  38, 44, 45, 47, 50, 52, 53, 54, 55, 56, 59, 60, 61, 62, 63, 64, 65, 66, 68, 69, 70, 71, \\
& 72, 73, 74, 75, 76, 77       \\
15 &  46, 48, 53, 54, 56, 57, 60, 62, 63, 64, 65, 66, 67, 70, 71, 72, 
      73, 74, 75, 76, 77, 78,\\ 
      &  80, 81, 82, 83, 84, 85, 86, 87, 88, 89, 90 \\
 16 &  47, 55, 58, 63, 64, 66, 67, 68, 71, 73, 74, 75, 76, 77, 78, 79, 
      82, 83, 84, 85, 86, 87,\\
      & 88, 89, 90, 91, 93, 94, 95, 96, 97, 98, 99, 
      100, 101, 102, 103, 104 \\\hline
\end{array}$$
\caption{The values of $c$ such that the 
co-degree sequence $\gamma_c$ cannot be realized by a tree, for $3 \leqslant  v \leqslant  16$ and $v-2 \leqslant c
\leqslant \binom{v-1}{2}$.}\label{tab:c}
\end{table}

We would like to understand which are the values of $c$ that can be realized by a tree. The following result provides an answer for the largest values. 

\begin{prop}\label{no tree}
Assume $v\geqslant 6$. 
There is no tree $T$ on $v$ vertices such that
$$\c(T)\in \left[\binom{v-3}{2}+4, \binom{v-2}{2} \right]\cup 
\left[ \binom{v-2}{2}+2, \binom{v-1}{2}-1\right].$$
\end{prop}

\begin{proof}
Let $\alpha$ be a partition of  $v-2$.
Applying Lemma~\ref{domstretta}, one can easily see that, if 
$$\alpha \not \in \{ (1^{v-2}), (2,1^{v-4}),(2^2,1^{v-6}),
  (v-4,2),(v-3,1),(v-2)\},$$ then 
\begin{equation}\label{prec}
(1^{v-2})\prec (2,1^{v-4})\prec(2^2,1^{v-6})\lhd \alpha \lhd (v-4,2)\prec (v-3,1)\prec(v-2).
\end{equation}
On the other hand, we have the following values for $\bm{c}(\alpha)$:
$$\begin{array}{|c|c|}\hline
 \alpha &  \bm{c}(\alpha)\\\hline & \\[-10pt]
   (v-2)  & (v^2-3v+2)/2=\binom{v-1}{2}\\[2pt]
   (v-3,1) & (v^2-5v+8)/2=\binom{v-2}{2}+1\\[2pt]
   (v-4,2) & (v^2-7v+18)/2=\binom{v-3}{2}+3\\
   (2^2, 1^{v-6}) & v \\
   (2,1^{v-4}) & v-1 \\
   (1^{v-2}) & v-2 \\\hline
   \end{array}$$
The statement now follows from \eqref{prec} and 
Lemma~\ref{part}.
\end{proof}

The argument exploited in the proof is not effective for smaller values of $c$.

In the next section we will consider the more general question of solving Problem~\ref{problema} for $C_4$-free graphs.

\section{Non-admissible values for the function \texorpdfstring{$\c$}{c}}\label{planari}

The function $\c: \mathcal{Q}_{v} \to \left[0, \binom{v}{2} \right]$ is not surjective, so we would like to determine its image.
Using \textsc{magma} and \textsc{nauty} we obtain the values of $c \in \left[0,\binom{v}{2} \right]$ such that  $\gamma_c$ is not the co-degree sequence of a simple graph on $v$ vertices for $2\leqslant v \leqslant 16$, see Table~\ref{vQv}.
Clearly, the set $\mathcal{Q}_v$ contains the set of all trees on $v$ vertices. So,
we first consider the values of $c$ provided by Theorem~\ref{Mic} concerning trees.
The next result allows us to ``fill the gaps'' of Lemma~\ref{no tree} and some of Table~\ref{tab:c}.

\begin{table}[htbp]
\begin{center}
\begin{tabular}{|c|c||c|c|}\hline
$v$ & $|\mathcal{Q}_v|$ & Missing values of $c$ \\\hline
 $2$ & \numprint{2}  & $1$   \\
$3$  & \numprint{4}  & $2$ \\
$4$ & \numprint{8} & $4,6$  \\
$5$ & \numprint{18}&  $9$ \\
$6$ & \numprint{44} &  $11,13,15$  \\
$7$ & \numprint{117} &  $18,20$\\
$8$ & \numprint{351} & $24,26,28$ \\
$9$ & \numprint{1230} & $31, 33, 35$ \\
$10$ & \numprint{5069} & $37, 39, 41, 43, 45$\\
$11$ & \numprint{25181} & $48, 50, 52, 54$ \\
$12$ &  \numprint{152045} & $56, 58, 60, 62, 64, 66$ \\
$13$ & \numprint{1116403} & $69, 71, 73, 75, 77$ \\
$14$ & \numprint{9899865} & $79, 81, 83, 85, 87, 89, 91$\\
$15$ & \numprint{104980369} & $94, 96, 98, 100, 102, 104$\\
$16$ & \numprint{1318017549} & 
$106, 108, 110, 112, 114, 116, 118, 120$\\\hline
\end{tabular}
\end{center}
\caption{The values of $c \in \left[0,\binom{v}{2} \right]$ such that  $\gamma_c$ is not the co-degree sequence of a simple graph on $v$ vertices for $2\leqslant v \leqslant 16$.}\label{vQv}
\end{table}

\begin{prop}\label{daisy}
Let $v\geqslant 4$. For any $c \in \left[v-2,\binom{v-1}{2} \right]$, there exists a (connected planar) graph $G \in \mathcal{Q}_v$ such that $\c(G)=c$. 
\end{prop}

\begin{proof}
For $v=4$ it suffices to consider the path and the star graph on $4$ vertices (these two graphs give the values $c=2,3$, respectively, as required).
So, assume $v\geqslant 5$.
Given $2\leqslant a\leqslant v-2$, let
$$E_a=\{\{1,i\}: i \in [2,a+1]\}\cup\{\{a+k,a+k+1\}: k \in [1,v-a-1]\}.$$
For any  $r \in \left[0, \left\lfloor \frac{a-1}{2}\right\rfloor\right]$, let $A_{a,r}$ be the simple graph whose vertex set is $[1,v]$ and whose edge set is 
$E_a\cup \{\{2j,2j+1\}:  j \in [1,r]\}$.
For any  $s \in \left[0, \left\lfloor \frac{a-2}{2}\right\rfloor\right]$, let $B_{a,s}$ be the simple graph whose vertex set is $[1,v]$ and whose edge set is 
$E_a\cup \{\{a,a+1\}\}\cup \{\{2j,2j+1\}:  j \in [1,s]\}$.

Then $A_{a,r}$ and $B_{a,s}$ are connected $C_4$-free graphs on $v$ vertices, 
$\c(A_{a,r})=\binom{a}{2}+v-1-a +2r = \binom{a-1}{2}+v-2+2r$
and $\c(B_{a,s})=\binom{a}{2}+ v-1-a  +3+2s=\binom{a-1}{2} + v+1+2s$.

Set 
$$\begin{array}{rcl}
\Theta_v & =& \left\{
\c(A_{a,r}): a \in [2,v-2], \; 
r \in \left[0, \left\lfloor \frac{a-1}{2}\right\rfloor\right] \right\}\\[2pt]
&& \cup\left\{
\c(B_{a,s}): a \in [2,v-2],\;
s \in \left[0, \left\lfloor \frac{a-2}{2}\right\rfloor\right] \right\}
\end{array}
$$
and $\Theta_v'=\{\vartheta-(v-2): \vartheta \in \Theta_v\}$.

Taking, for all $0\leqslant b\leqslant \left\lfloor\frac{v-4}{2} \right\rfloor$,
$$\begin{array}{rcl}
\mathcal{A}_{2b+2} & =& \{\c(A_{2b+2,r})-(v-2):
0\leqslant r\leqslant b \}, \\
\mathcal{A}_{2b+3} & =& \{\c(A_{2b+3,r})-(v-2):
0\leqslant r\leqslant b+1 \},\\
\mathcal{B}_{2b+2} & =& \{\c(B_{2b+2,r})-(v-2):
0\leqslant r\leqslant b \}, \\
\mathcal{B}_{2b+3} & =& \{\c(B_{2b+3,r})-(v-2):
0\leqslant r\leqslant b \},
\end{array}$$
we have
$$\begin{array}{rclcl}
\mathcal{A}_{2b+2} & =& \left[
\binom{2b+1}{2},
\binom{2b+1}{2}+2b \right]_2 & = & \left[
\binom{2b+1}{2},
\binom{2b+2}{2}-1 \right]_2,\\[2pt]
\mathcal{A}_{2b+3} & =& \left[
\binom{2b+2}{2},
\binom{2b+2}{2}+2b+2\right]_2 & =& 
\left[
\binom{2b+2}{2},
\binom{2b+3}{2}\right]_2 ,\\[2pt]
\mathcal{B}_{2b+2} & =& \left[
\binom{2b+1}{2} + 3,
\binom{2b+1}{2} + 2b+3
\right]_2 & =& 
\left[
\binom{2b+1}{2} + 3,
\binom{2b+2}{2} + 2
\right]_2, \\[2pt]
\mathcal{B}_{2b+3} & =& \left[
\binom{2b+2}{2} + 3,
\binom{2b+2}{2} +2b+3 \right]_2 & =& 
\left[
\binom{2b+2}{2} + 3,
\binom{2b+3}{2} +1 \right]_2.
\end{array}$$
Hence, 
$\mathcal{A}_2\cup \mathcal{B}_2\cup\mathcal{A}_3\cup \mathcal{B}_3=\{0,1,3,4 \}$
and, for $b>0$,
$$\mathcal{A}_{2b+2}\cup \mathcal{B}_{2b+2}=
\left\{\binom{2b+1}{2} \right\}\cup
\left[\binom{2b+1}{2}+2,
\binom{2b+2}{2}\right]\cup
\left\{ \binom{2b+2}{2}+2 \right\}$$
and
$$\mathcal{A}_{2b+3}\cup \mathcal{B}_{2b+3}=
\left\{\binom{2b+2}{2} \right\}\cup
\left[\binom{2b+2}{2}+2,
\binom{2b+3}{2}+1\right].$$

Suppose $v$ odd and write $v=2u+5$ with $u\geqslant 0$.
Taking $a\in \{2b+2,2b+3\}$, condition $2\leqslant a \leqslant v-2$ becomes $0\leqslant b\leqslant u $. 
Hence, 
$$\begin{array}{rcl}
\Theta_{2u+5}' & =& \bigcup\limits_{b=0}^{u}
\left(\mathcal{A}_{2b+2}\cup \mathcal{B}_{2b+2} \cup \mathcal{A}_{2b+3}\cup 
\mathcal{B}_{2b+3}\right)\\[2pt]
& =& \{0,1,3,4 \}\cup
\bigcup\limits_{b=1}^{u}
\left(\left[\binom{2b+1}{2},\binom{2b+3}{2}+1\right] \setminus\left\{ \binom{2b+1}{2}+1, 
\binom{2b+2}{2}+1\right\} \right)\\[2pt]
 &=& \left[0,\binom{v-2}{2}+1  \right]
 \setminus  \left\{\binom{2j+2}{2}+1: j \in [0,u] \right\}.
\end{array}$$ 
Note that  $u=\frac{v-5}{2}=\left\lfloor \frac{v-4}{2}\right\rfloor$.

Suppose $v$ even and write $v=2u+6$ with $u\geqslant 0$.
Condition $2\leqslant a \leqslant v-2$ becomes $0\leqslant b\leqslant u+1$ for $a=2b+2$ and $0\leqslant b\leqslant u $ for $a=2b+3$.
Hence, 
$$\begin{array}{rcl}
\Theta_{2u+6}' & =& \bigcup\limits_{b=0}^{u}
\left(\mathcal{A}_{2b+2}\cup \mathcal{B}_{2b+2} \cup \mathcal{A}_{2b+3}\cup 
\mathcal{B}_{2b+3}\right)
\cup \left( \mathcal{A}_{2(u+1)+2}\cup \mathcal{B}_{2(u+1)+2} \right)\\[2pt]
&& \cup \left\{\binom{2u+3}{2} \right\}\cup
\left[\binom{2u+3}{2}+2,
\binom{2u+4}{2}\right]\cup
\left\{ \binom{2u+4}{2}+2 \right\}
\\[2pt]
& =& \{0,1,3,4 \}\cup
\bigcup\limits_{b=1}^{u+2}
\left(\left[\binom{2b+1}{2},\binom{2b+3}{2}+1\right] \setminus\left\{ \binom{2b+1}{2}+1, 
\binom{2b+2}{2}+1\right\} \right)\\[2pt]
&& \cup \left\{\binom{v-3}{2} \right\}\cup
\left[\binom{v-3}{2}+2,
\binom{v-2}{2}\right]\cup
\left\{ \binom{v-2}{2}+2 \right\}
\\[2pt]
 &=& \left[0,\binom{v-2}{2}+2  \right]
 \setminus  \left\{\binom{2j+2}{2}+1: j \in [0,u+1] \right\}.
\end{array}$$
Note that  $u+1=\frac{v-4}{2}
=\left\lfloor \frac{v-4}{2}\right\rfloor$.

Now, for any $v\geqslant 6$ and any $j \in \left[0,\left\lfloor\frac{v-6}{2}\right\rfloor\right]$ take the canonical tree $T_j$ associated to the partition 
$$\alpha_j=\left(2j+2, 2, 1^{v-6- 2j} \right)$$
of $v-2$.
Then $\c(T_j)=\bm{c}(\alpha_j)= \binom{2j+2}{2}+ 1+(v-2)$. 
Writing 
$$\Upsilon_v=\left\{
\c(T_j): j \in \left[0,\left\lfloor\frac{v-4}{2}\right\rfloor-1\right]
\right\},$$
we have
$\Theta_{2u+5}\cup \Upsilon_{2u+5}
=\left[0,\binom{v-1}{2}+1 \right] \setminus\left\{
\binom{v-2}{2}+2
\right\}$
and
$\Theta_{2u+6}\cup \Upsilon_{2u+6}=\left[0,\binom{v-1}{2} \right]$.

Finally, the statement follows by considering, for $v\geqslant 5$ odd, the simple graph $L_v$  whose vertex set is  $[1,v]$ and whose edge set is: 
$$E_{v-4}\cup \{ \{1,v\}\} \cup \left\{ \{2j,2j+1\}: j \in \left[1,\frac{v-5}{2}\right] \right\}. $$
In this way, $\c(L_v)=\binom{v-3}{2} + 4 + v-5 =  \binom{v-2}{2} + 2$.
\end{proof}

Now, we prove some results linking the connectedness of a graph $G$ and  the values $\c(G)$ and $\Delta(G)$. The first lemma is a sort of reverse of 
Proposition~\ref{v-2}.

\begin{lem}\label{connesso}
Assume $v\geqslant 4$. Let $G\in \mathcal{Q}_v$  be such that
$\c(G) > \binom{v-1}{2}$.
Then $G$ is connected.
\end{lem}

\begin{proof}
Let $A_1,\ldots,A_\ell$ be the connected components of $G$; set  $a_i=|A_i|$, for each $i=1,\ldots,\ell$.
For the sake of contradiction, suppose $\ell\geqslant 2$.
For each pair of vertices $x,y$ such that $x \in A_i$ and $y \in A_j$ with $i\neq j$ we have $c_{x,y}=0$. So, the sequence $\ga(G)$ contains at least $Z:=\sum\limits_{i=1}^{\ell-1} \sum\limits_{j=i+1}^\ell a_ia_j$ zeros.
Now, $Z\geqslant \sum\limits_{i=1}^\ell a_i=v$ implies that $\c(G)\leqslant \binom{v}{2}-v=\binom{v-1}{2}-1$, a contradiction.
\end{proof}

\begin{lem}\label{binomiale}
Let $G=(V,E)$ be a graph in $\mathcal{Q}_v$. For every $x \in V$ one has
$$\binom{v-\deg_G(x)}{2} \geqslant \sum_{y \in V\setminus \{x\}} \binom{\deg_G(y)-1}{2}.$$
\end{lem}

\begin{proof}
The result can be obtained by following the proof of \cite[Lemma]{F}; see also \cite{MSE}.
\end{proof}

\begin{lem}\label{v-3}
Let $G \in \mathcal{Q}_v$, $v\geqslant 9$. If $\c(G)>\binom{v-1}{2}$, then $\Delta(G) \neq v-3$.
\end{lem}

\begin{proof}
For the sake of contradiction, suppose there exists a graph $G=(V,E)$ in $\mathcal{Q}_v$ such that $\c(G)> \binom{v-1}{2}$ and $\Delta(G)=v-3$.
Then $G$ is connected by Lemma~\ref{connesso}.
Let $x$ be a vertex of $G$ having degree $v-3$: write $N_G(x)=\{y_1,\ldots ,y_{v-3}\}$ and $V\setminus N_G(x)=\{z_1,z_2\}$.

Since $G$ is a $C_4$-free graph, each vertex $y_i$ can be adjacent to at most one other vertex $y_{i'}$ and each vertex $z_j$ can be adjacent to at most one vertex $y_i$: hence, 
$y_1,\ldots,y_{v-3}$ have degree at most $4$, but at most two of them  can have degree greater than $2$;
also, $z_1,z_2$ have degree at most $2$. 

By Lemma~\ref{binomiale}, we have
$$\binom{3}{2}=\binom{v-\deg_G(x)}{2} \geqslant 
\sum_{i=1}^{v-3} \binom{\deg_G(y_i)-1}{2}
+\sum_{j=1}^{2} \binom{\deg_G(z_j)-1}{2}.$$ It follows that $\sum\limits_{i=1}^{v-3} \binom{\deg_G(y_i)-1}{2}\leqslant 3$. 
Then, we have the following cases: $(i)$ all the vertices $y_1,\ldots,y_{v-3}$ have degree at most $2$;
$(ii)$ one or two vertices, say  $y_1,y_2$, have degree $3$ and the vertices $y_{3},\ldots,y_{v-3}$ have degree at most $2$;
$(iii)$  one vertex, say $y_1$, has degree $4$ and the other vertices $y_2,\ldots,y_{v-3}$ have degree at most $2$.
On the other hand, we have $\c(G)=\binom{v-3}{2}+\sum\limits_{i=1}^{v-3} \binom{\deg_G(y_i)}{2}
+\sum\limits_{j=1}^{2} \binom{\deg_G(z_j)}{2}> \binom{v-1}{2}$ and this excludes case $(i)$. In case $(ii)$ we obtain
$$\binom{v-1}{2}+1\leqslant \c(G)\leqslant \binom{v-3}{2} + 2\binom{3}{2}+(v-3)\binom{2}{2} =
\binom{v-1}{2}+8-v,$$
which gives $v\leqslant 7$, an absurd.
Finally, in case $(iii)$
from $$\binom{v-1}{2}+1\leqslant \c(G)\leqslant \binom{v-3}{2} + \binom{4}{2}+(v-2)\binom{2}{2} =
\binom{v-1}{2}+9-v$$
we obtain $v\leqslant 8$, an absurd.
\end{proof}

\begin{lem}\label{deltavmezzi}
    Assume $v \geqslant 4$. Let $G\in \mathcal{Q}_v$ be such that $\c(G) > \binom{v-1}{2}$ and $\Delta(G) < v-3$. Then $\Delta(G) \leqslant\frac{v+1}{2}$.
\end{lem}

\begin{proof}
    Let $x \in V(G)$ be a vertex with $\deg_G(x)=\Delta(G)$.
    Write $\Delta(G)=v-\lambda$ and observe
    that $\Delta(G) < v-3$ implies $\lambda>3$ and hence $\frac{2}{\lambda-2} \leqslant 1$.

    Since $\Delta(G) \neq v-1$, the set $V(G)\setminus(N_G(x) \cup\{x\})$ contains a vertex $y$. By the $C_4$-free hypothesis, $y$ has at most one neighbor in $N_G(x)$ and has at most $\lambda-2$ neighbors in $V(G)\setminus(N_G(x) \cup\{x\})$. Moreover, every vertex in $N_G(x)$ has at most one neighbor in $N_G(x)$ itself.
    Then, for each $y \in V(G)\setminus(N_G(x) \cup\{x\})$, we have $|\{z\in N_G(x) : c_{z,y}=1\}| \leqslant 1+ \lambda -2 = \lambda-1$ and hence $|\{z\in N_G(x) : c_{z,y}=0\}| \geqslant v - \lambda - (\lambda-1)=v-2\lambda + 1$.

    It follows that $\ga(G)$ contains at least 
    $|V(G)\setminus(N_G(x) \cup\{x\})| \cdot (v-2\lambda + 1) = (\lambda-1)(v-2\lambda + 1)$ zeros.
     Hence, $\binom{v-1}{2}+1 \leqslant \c(G) \leqslant \binom{v}{2}-(\lambda-1)(v-2\lambda + 1)$, which gives
    $$(\lambda -2 )v \leqslant 2\lambda^2-3\lambda-1
    =(2\lambda+1)(\lambda-2)+1$$ and then
    $v\leqslant 2\lambda+1$. We conclude that 
    $\Delta(G) \leqslant \frac{v+1}{2}$. 
\end{proof}

\begin{prop}\label{propDelta}
    Assume $v \geqslant 4$. Let $G\in \mathcal{Q}_v$ be such that $\c(G) > \binom{v-1}{2}$ and $\Delta(G) < v-3$. 
    Then $\left\lceil \sqrt{v}\right\rceil \leqslant \Delta(G) \leqslant\frac{v+1}{2}$.
\end{prop}
\begin{proof}
From
$\c(G)\leqslant v\binom{\Delta(G)}{2}$ and $\c(G)\geqslant \binom{v-1}{2}+1$ we obtain the condition
$$\Delta(G)^2-\Delta(G)- \frac{v^2-3v+4}{v} \geqslant 0,$$
whence 
$\Delta(G)\geqslant \left\lceil\frac{1}{2}+\sqrt{\frac{4v^2-11v+16}{4v}}\right\rceil$ and we  conclude that $\Delta(G)\geqslant \left\lceil \sqrt{v}\right\rceil$.

The statement now follows from Lemma \ref{deltavmezzi}.
\end{proof}

\begin{prop}\label{planare}
 Assume $v\geqslant 17$.  Let $G=(V,E)$ be a \emph{planar} graph in  $\mathcal{Q}_v$. If $\c(G)> \binom{v-1}{2}$, then $\Delta(G)\geqslant v-2$. 
\end{prop}

\begin{proof}
Denoting $e=|E(G)|$, we have
$$\c(G)=\sum_{y \in V} \binom{\deg_G(y)}{2} =
\sum_{y \in V} \left(\binom{\deg_G(y)-1}{2} + (\deg_G(y) -1)\right),$$
whence
$$2e = \c(G)+v  -\sum_{y \in V} \binom{\deg_G(y)-1}{2}.$$   

Now, fix a vertex $x\in V$ of degree $d$. By Lemma~\ref{binomiale}, it follows that $2e\geqslant \c(G) +v - \binom{v-d}{2}-\binom{d-1}{2}$ and hence
$$2e> \binom{v-1}{2} +v - \binom{v-d}{2}-\binom{d-1}{2}=d(v+1)-d^2.$$

Since $G$ is a $C_4$-free  planar graph, by \cite[Theorem 1]{D} we have $e\leqslant \frac{15}{7}(v-2)$. This implies $d(v+1)-d^2 < \frac{30}{7}(v-2)$, that is $7d^2 -7(v+1)d+ 30(v-2)>0$.
This gives
$$ d \leqslant  \left\lfloor \frac{7(v+1)-\sqrt{49v^2-742v+1729}}{14}\right\rfloor \quad \text{or}\quad  d\geqslant  \left\lceil \frac{7(v+1)+\sqrt{49v^2-742v+1729}}{14}\right\rceil. $$
Since $v\geqslant 17$ we have $(7v-62)^2< 49v^2-742v+1729< (7v-49)^2$, whence
$$4\leqslant  \left\lfloor \frac{7(v+1)-\sqrt{49v^2-742v+1729}}{14}\right\rfloor\leqslant 5- \frac{1}{14}$$
and 
$$(v-4) +\frac{1}{14} \leqslant  \left\lceil\frac{7(v+1)+\sqrt{49v^2-742v+1729}}{14}\right\rceil\leqslant  v-3.$$
This means that either $d\leqslant 4$ or $d\geqslant v-3$.

So, suppose that $\deg_G(x)\leqslant 4$ for all $x \in V$. Then $\binom{v-1}{2} < \c(G)\leqslant\binom{4}{2}v$ gives the inequality $v^2-15v+2<0$, whence the absurd $v \leqslant 14$. 
We conclude that there is at least one vertex $x$ such that $\deg_G(x)\geqslant v-3$. The statement now follows from Lemma~\ref{v-3}.
\end{proof}

\begin{cor}\label{utopia}
Assume $v\geqslant 5$. Let $G\in \mathcal{Q}_v$  be such that
$\c(G) > \binom{v-1}{2}$.
Then one of the following cases occurs:
\begin{itemize}
    \item[(1)] $\Delta(G)=4$, $v=7$ and 
    $G$ is one of the graphs shown in Figure~\ref{7};
    \item[(2)] $\Delta(G)=4$, $v=8$ and $G$ is the graph shown in Figure~\ref{22};
    \item[(3)] $\Delta(G)\geqslant v-2$ (hence, $G$ is planar);
    \item[(4)] 
    $v\geqslant 17$, $\left\lceil \sqrt{v}\right\rceil \leqslant \Delta(G) \leqslant \frac{v+1}{2}$ and $G$ is not planar.
\end{itemize}
\end{cor}

\begin{proof}
The validity of the statement has been verified using \textsc{magma} and \textsc{nauty}  for $5\leqslant v\leqslant 16$.  
So, assume $v\geqslant 17$. If $G$ is planar, then $\Delta(G)\geqslant v-2$ by Proposition~\ref{planare}.
So, assume that $G$ is not planar. 
Then, $\Delta(G)\leqslant v-3$ and hence
$\left\lceil \sqrt{v}\right\rceil \leqslant \Delta(G) \leqslant(v+1)/2$ by Lemma~\ref{v-3} and Proposition~\ref{propDelta}.
\end{proof}

We conjecture that there is no graph satisfying case $(4)$ of the previous corollary.

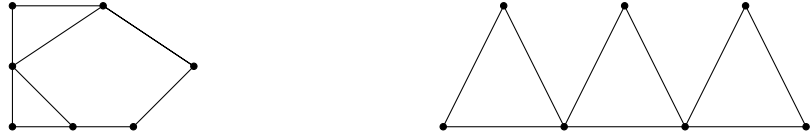
\begin{figure}[htbp]
\centering

\begin{subfigure}{.45\linewidth}
\centering

\begin{tikzpicture}[scale = 0.4]

\coordinate (A) at (0,0);
\coordinate (B) at (0,2);
\coordinate (C) at (0,4);
\coordinate (D) at (3,4);
\coordinate (F) at (6,2);
\coordinate (G) at (2,0);
\coordinate (H) at (4,0);

\draw (F)--(H)--(G)--(A)--(B)--(C)--(D)--(F)--(D)--(B)--(G);

\draw[fill] (A) circle (3pt);
\draw[fill] (B) circle (3pt);
\draw[fill] (C) circle (3pt);
\draw[fill] (D) circle (3pt);
\draw[fill] (F) circle (3pt);
\draw[fill] (G) circle (3pt);
\draw[fill] (H) circle (3pt);
\end{tikzpicture}
\end{subfigure}
\quad
\begin{subfigure}{.45\linewidth}
\centering
\begin{tikzpicture}[scale = 0.4]
\centering

\coordinate (A) at (0,0);
\coordinate (B) at (4,0);
\coordinate (C) at (8,0);
\coordinate (D) at (12,0);
\coordinate (E) at (2,4);
\coordinate (F) at (6,4);
\coordinate (G) at (10,4);

\draw (A)--(B)--(C)--(D)--(G)--(C)--(F)--(B)--(E)--(A);

\draw[fill] (A) circle (3pt);
\draw[fill] (B) circle (3pt);
\draw[fill] (C) circle (3pt);
\draw[fill] (D) circle (3pt);
\draw[fill] (E) circle (3pt);
\draw[fill] (F) circle (3pt);
\draw[fill] (G) circle (3pt);
\end{tikzpicture}

\end{subfigure}

\caption{On the left, a graph $G_1 \in \mathcal{Q}_7$ with $\c(G_1)=16$ and $\Delta(G_1)=4$; on the right, 
a graph $G_2 \in \mathcal{Q}_7$ with $\c(G_2)=17$ and $\Delta(G_2)=4$.}\label{7}
\end{figure}

\begin{figure}[ht]
\centering

\begin{tikzpicture}[scale = 0.4]

\coordinate (A) at (0,0);
\coordinate (B) at (0,2);
\coordinate (C) at (0,4);
\coordinate (D) at (3,4);
\coordinate (E) at (6,4);
\coordinate (F) at (6,2);
\coordinate (G) at (2,0);
\coordinate (H) at (4,0);

\draw (F)--(H)--(G)--(A)--(B)--(C)--(D)--(E)--(F)--(D)--(B)--(G);

\draw[fill] (A) circle (3pt);
\draw[fill] (B) circle (3pt);
\draw[fill] (C) circle (3pt);
\draw[fill] (D) circle (3pt);
\draw[fill] (E) circle (3pt);
\draw[fill] (F) circle (3pt);
\draw[fill] (G) circle (3pt);
\draw[fill] (H) circle (3pt);
\end{tikzpicture}
\caption{A graph $G \in \mathcal{Q}_8$ with $\c(G)=22$ and $\Delta(G)=4$.}\label{22}
\end{figure}
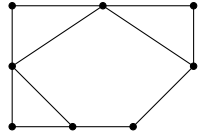

\begin{thm}\label{mainPla}
Assume $v\geqslant 4$.
For every $0\leqslant c \leqslant \binom{v}{2}$, there exists a \emph{planar} graph $G \in \mathcal{Q}_v$ such that $c=\c(G)$, except when
\begin{itemize}
\item[$(1)$] $v=8$ and $c \in \{ 24, 26, 28 \}$;
    \item[$(2)$] $v\neq 8$ is even and 
    $$c\in \left\{\binom{v-1}{2}+1, \binom{v-1}{2}+3,\binom{v-1}{2}+5,\ldots, \binom{v}{2}  \right\};$$ 
        \item[$(3)$] $v$ is odd and 
    $$c\in \left\{\binom{v-1}{2}+3, \binom{v-1}{2}+5,\binom{v-1}{2}+7,\ldots, \binom{v}{2}-1  \right\}.$$ 
\end{itemize}
\end{thm}

\begin{proof}
By Proposition~\ref{v-2}, if a graph $G \in \mathcal{Q}_v$ is such that
$\c(G)\in [0,v-3]$, then $G$ cannot be connected.
On the other hand, for all $\ell \in [2,v-1]$, take the graph $H_\ell$ consisting of $v-\ell$ isolated vertices and a path with $\ell$ vertices. Then, $H_\ell \in\mathcal{Q}_v$ and $\c(H_\ell)=\ell-2$.
For $c \in \left[v-2, \binom{v-1}{2}\right]$ the statement follows from Proposition~\ref{daisy}.

We now consider the case when $c > \binom{v-1}{2}$.
For any $\ell \in \left[0, \left\lfloor \frac{v-1}{2}\right\rfloor\right]$, 
let $D_{v,\ell}$ be the graph whose vertex set is  $[1,v]$ and whose edge set is 
$\{\{1,i\}: i \in [2,v]\}  \cup \{\{2j,2j+1\}:  j \in [1,\ell]\}$. 
Then $D_{v,\ell}\in \mathcal{Q}_v$ is connected and 
$\c(D_{v,\ell})=\binom{v-1}{2}+2\ell$.
Setting 
$$\mathcal{D}_{v}  = \left\{\c(D_{v,\ell}): \ell \in  \left[0, \left\lfloor \frac{v-1}{2}\right\rfloor\right]\right\},$$
if $v$ is even then
   $$
   \mathcal{D}_v  =\left\{\binom{v-1}{2}, \binom{v-1}{2}+2, \ldots, \binom{v}{2}-1  \right\};$$
if $v$ is odd then
   $$   \mathcal{D}_v = \left\{\binom{v-1}{2}, \binom{v-1}{2}+2, \ldots, \binom{v}{2} \right\}.$$
In particular,  one can verify that there is no graph
$G\in \mathcal{Q}_4$ such that $\c(G)\in \{4,6\}$ (see Figure~\ref{4}).

So, assume $v\geqslant 5$
and suppose that $G \in \mathcal{Q}_v$ is such that $\c(G)> \binom{v-1}{2}$.
We apply Corollary~\ref{utopia}:
if $v=7$ and $G$ is one of the graphs shown in Figure~\ref{7}, then $\c(G)\in \left\{ \binom{6}{2}+1, \binom{6}{2}+2 \right\}$;
if $v=8$ and $G$ is  the graph shown in Figure~\ref{22}, then $\c(G)=22=\binom{7}{2}+1$.
In the other cases, $\Delta(G)\geqslant v-2$.

If $\Delta(G)=v-1$, it is easy to see that $G$ must be a graph 
$D_{v,\ell}$, previously described.
If $\Delta(G)=v-2$, then we have two cases: either $G$ is a graph
$A_{v-2,r}$ for some $r \in \left[0,\left\lfloor \frac{v-3}{2}\right\rfloor\right]$ 
or a graph $B_{v-2,s}$ for 
some $s \in \left[0,\left\lfloor \frac{v-4}{2} \right\rfloor\right]$, graphs defined in the proof of Proposition~\ref{daisy}.
In the first case, $\c(A_{v-2,s})
=\binom{v-3}{2} + v-2 +2r > \binom{v-1}{2}$ implies $r> \frac{v-3}{2}$, a contradiction.
In the second case, $\c(B_{v-2,s})=\binom{v-3}{2}+v+1+2s > \binom{v-1}{2}$ implies $s > \frac{v-6}{2}$.
We conclude that the only possibilities are 
$G=B_{v-2,\frac{v-5}{2}}$ if $v$ is odd and 
$G=B_{v-2,\frac{v-4}{2}}$ if $v$ is even. In the first case we have
$\c(G)=\binom{v-1}{2}+1$;
in the second case, 
$\c(G)=\binom{v-1}{2}+2$.
\end{proof}


\section{Conclusions}\label{sec:conclusioni}

In this paper, we proposed a reconstruction problem for simple graphs
based on the concept of co-degree sequence. Our main investigation was focused on trees  (see Theorem~\ref{Mic}) and $C_4$-free graphs. 
In particular, 
Theorem~\ref{mainPla} describes which are the co-degree sequences for planar $C_4$-free graphs. We conjecture that the statement of this theorem holds also if we remove the hypothesis of planarity. In fact, we believe that case $(4)$ of Corollary~\ref{utopia} does not occur.

\section*{Acknowledgments}

Some of the results of this paper are part of first author's PhD thesis carried out under the supervision of Prof.~Anusch Taraz.
The second, third and fourth authors are members of INdAM-GNSAGA.

\end{document}